%% file: paper_new.tex
\documentclass[review,onefignum,onetabnum]{siamart190516}

\input{shared}

\newcommand{\Ctwo}{\ensuremath{\mathcal{C}^2 }}
\ifpdf
\hypersetup{
	pdftitle={Local convergence of sequential quadratic programming for nonsmooth optimization problems with upper-$\mathcal{C}^2$ objective},
  pdfauthor={J. Wang and C. G. Petra}
}
\fi

\begin{document}
\nolinenumbers
\maketitle

% REQUIRED
\begin{abstract}
	A sequential quadratic programming (SQP) algorithm is designed for nonsmooth optimization problems with upper-$\Ctwo$ objective functions. Upper-$\Ctwo$ functions are locally equivalent to difference-of-convex (DC) functions with smooth convex parts. They arise naturally in many applications such as certain classes of solutions to parametric optimization problems, \textit{e.g.}, recourse of stochastic programming, and projection onto closed sets. 
	The proposed algorithm conducts line search and adopts an exact penalty merit function.
        The potential inconsistency due to the linearization of constraints are addressed through relaxation, similar to that of S$\ell_1$QP. 
We show that the algorithm is globally convergent under reasonable assumptions. 
        Moreover, we study the local convergence behavior of the algorithm under additional assumptions of Kurdyka-Łojasiewicz (KL) properties, which have been applied to many nonsmooth optimization problems. 
   Due to the nonconvex nature of the problems, a special potential function is used to analyze local convergence. 
     We show that under acceptable assumptions, upper bounds on local convergence can be proven.
        Additionally, we show that for a large number of optimization problems with upper-$\Ctwo$ objectives, their corresponding potential functions are indeed KL functions. Numerical experiment is performed with a power grid optimization problem that is consistent with the assumptions and analysis in this paper. 
\end{abstract}

% REQUIRED

% REQUIRED

\input{Sections_journal/Introduction.tex}

\input{Sections_journal/Twostage.tex}

\input{Sections_journal/Algorithm2.tex}

\input{Sections_journal/Algorithm_local.tex}

%\input{Sections_journal/Algorithm_sto.tex}

\input{Sections_journal/Examples.tex}

\input{Sections_journal/Conclusion.tex}

\appendix

\section*{Acknowledgments}
Prepared by LLNL under Contract DE-AC52-07NA27344. Release number: LLNL-JRNL-856526. 
\bibliographystyle{siamplain}
\bibliography{bibliography}
\end{document}

%% file: shared.tex
\usepackage{lipsum}
\usepackage{amsfonts}
\usepackage{graphicx}
\usepackage{epstopdf}
\usepackage[noend]{algorithmic}
\ifpdf
  \DeclareGraphicsExtensions{.eps,.pdf,.png,.jpg}
\else
  \DeclareGraphicsExtensions{.eps}
\fi

\newsiamremark{remark}{Remark}
\newsiamremark{hypothesis}{Hypothesis}
\crefname{hypothesis}{Hypothesis}{Hypotheses}
\newsiamthm{claim}{Claim}
\newtheorem{assumption}[theorem]{Assumption}

\headers{SQP for upper-$C^2$ objective}{J. Wang and C. G. Petra}

\title{Local convergence of a sequential quadratic programming method for a class of nonsmooth nonconvex objectives
        \thanks{Submitted to the editors.
\funding{Prepared by Lawrence Livermore National Laboratory (LLNL) under Contract DE-AC52-07NA27344.}}}

\author{Jingyi Wang\thanks{Center for Applied Scientific Computing, Lawrence Livermore National Laboratory,
Livermore, CA 
  (\email{wang125@llnl.gov, petra1@llnl.gov}).}
\and Cosmin G. Petra\footnotemark[2]
}

\usepackage{amsopn}

%% file: Sections_journal/Introduction.tex
%!TEX root = ../paper.tex

\newcommand{\Rbb}{\ensuremath{\mathbb{R} }}
\newcommand{\Nbb}{\ensuremath{\mathbb{N} }}
\newcommand{\Pbb}{\ensuremath{\mathbb{P} }}
\newcommand{\Ebb}{\ensuremath{\mathbb{E} }}
\section{Introduction}\label{se:intro}

In this paper, we consider the class of nonsmooth optimization problems in the form of
\begin{equation} \label{eqn:opt0}
 \centering
  \begin{aligned}
   &\underset{\substack{x}}{\text{minimize}} 
	  & & f(x)\\
   &\text{subject to}
	  & & c_i(x) = 0, \ i \in\mathcal{E} \\
	  &&& c_i(x) \geq 0, \ i \in \mathcal{I},
  \end{aligned}
\end{equation}
where functions $c_i:\Rbb^n\to\Rbb^{m_e},i\in\mathcal{E}$ and $c_i:\Rbb^n\to\Rbb^{m_i},i\in\mathcal{I}$ are smooth, \textit{i.e.}, continuously differentiable with Lipschitz gradients. 
The objective function $f:\Rbb^n\to\Rbb\cup\{+\infty\}$ is upper-$\Ctwo$ and thus locally Lipschitz~\cite[Theorem~10.31]{rockafellar1998}. 
The index sets $\mathcal{E}$ and $\mathcal{I}$ are finite. The feasible region is denoted by
\begin{equation} \label{eqn:feasible-set}
 \centering
  \begin{aligned}
	  \Omega = \{ x | c_i(x)=0, i\in\mathcal{E}; \ c_i(x)\geq 0, i\in\mathcal{I} \}.
  \end{aligned}
\end{equation}
%Problem~\eqref{eqn:opt0} can rewritten compactly as 
%\begin{equation} \label{eqn:opt0-set}
% \centering
%  \begin{aligned}
%   &\underset{\substack{x\in \Omega}}{\text{minimize}} 
%	  & & f(x).\\
%  \end{aligned}
%\end{equation}

Any finite, concave function is upper-$\Ctwo$~\cite[Theorem~10.33]{rockafellar1998}, as are continuously differentiable functions~\cite[Proposition~13.34]{rockafellar1998}.
Let $f$ be a DC function on $O\subset \Rbb^n$ which can be decomposed into 
\begin{equation} \label{eqn:f-dc}
 \centering
  \begin{aligned}
    f(x) = g(x) - h(x), 
  \end{aligned}
\end{equation}
where $g$ and $h$ are convex functions. If $g(\cdot)$ is smooth, then $f(\cdot)$ is upper-$\Ctwo$~\cite[10.35]{rockafellar1998}. 
Further, upper-$\Ctwo$ functions are locally equivalent to DC functions with smooth convex parts~\cite[Theorem~10.33]{rockafellar1998}. 

A squared distance function to a closed set is upper-$C^2$~\cite[Example~10.57]{rockafellar1998}.
The second-stage value function of a large number of two-stage (stochastic) programming problems with recourse~\cite{Shapiro_book,Birge97Book,KallWallace} is upper-$\Ctwo$. 
Our target application for the upper-$\Ctwo$ property is 
the distributed security-constrained alternating current optimal power flow (SCACOPF) problems~\cite{ChiangPetraZavala_14_PIPSNLP,Qiu2005,petra_14_augIncomplete,petra_14_realtime,petra_21_gollnlp}, which can be formed as a two-stage recourse problem with bounded feasible regions. 
In this case, the nonsmooth part of the first-stage objective $f(\cdot)$ becomes upper-$C^2$ through regularization of the second-stage problems~\cite{wang2022}. 
Moreover, the power grid optimization variables $x$ enjoy bound constraints, which makes the feasible region~\eqref{eqn:feasible-set} bounded.

Potential algorithms for solving~\eqref{eqn:opt0} include bundle methods~\cite{makela1992,Kiwiel1996,mifflin1982,kiwiel1985,hare2010,noll2013}, DC algorithms~\cite{an2018}, etc. 
Sequential quadratic programming (SQP), a popular optimization method for constrained problems, has been extended to nonsmooth optimization problems.
In~\cite{curtis2012}, SQP is combined with gradient sampling for nonsmooth inequality-constrained optimization with proof of global convergence.
A BFGS-SQP method is proposed in~\cite{curtis2017}, which shows promising convergence behaviors without requiring the existence of Hessian.
The constrained DC algorithms (DCA), which in essence linearizes the concave part of the objective, 
bears similarity to SQP~\cite{an2018,cui2021} and has been shown to converge to critical points. 
In~\cite{wang2022}, the authors proposed a trust-region updated SQP with line search on the constraints. The algorithm is shown to enjoy subsequential convergence for optimization problems in the form of~\eqref{eqn:opt0} under proper assumptions.  
In this paper, we aim to further simplify the SQP method to the classic line search algorithm with the $\ell_1$ merit function.
Additionally, the constraint qualification is relaxed from  linear independent constraint qualification (LICQ) to  Mangasarian-Fromovitz constraint qualification (MFCQ).

%In dealing with nonsmooth nonconvex objectives with general constraint, ideas from penalty and filter methods are often applied to incorporate the constraint into the objective in bundle methods. The global convergence studies in this case typically shows the algorithm can either converge to a KKT point of the original problem or to a stationary/critical point of the constraint violation. The latter case could lead to an infeasible solution.
%An exact penalty merit function that measures the progress of both the objective and inequality constraint violation is used in the redistributed bundle method in~\cite{yang2014}.
%Lower-$C^2$ and a special strong Slater condition are assumed to ensure global convergence.
%In~\cite{dao2015}, a progress function that is the maximum of a penalized objective reduction and constraint violation is chosen. The bundle method is applied to the subproblem whose objective is replaced by the progress function, eliminating the general constraint. Given lower-$C^1$ or upper-$C^1$ property, convergence is proved. Such assumptions on the objective functions are the least restrictive known to the authors. Similar algorithm with direct assumptions on the penalty and quadratic parameters are presented in~\cite{monjezi2019}. Others have chosen a different set of penalized objective functions~\cite{lv2018}.

While the global convergence behavior of optimization methods for nonsmooth nonconvex problems has been studied extensively, their local convergence results are still challenging to establish. 
It is well-known that smooth quasi-Newton SQP can achieve superlinear local convergence if conditions on the Hessian approximation, second-order constraint qualification and strict complementarity are met~\cite{Nocedal_book}.  
Further, the quasi-Newton S$\ell_1$QP method could maintain superlinear convergence if carefully designed second-order correction terms are used~\cite{boggs1995sqp,fukushima1986secondorder}.
For nonsmooth objectives however, the Hessian and the conditions related to it are no longer available. 
Recently, a number of assumptions based on the Kurdyka-Łojasiewicz (KL) property have enjoyed huge success in ensuring linear local convergence of first-order methods for nonsmooth (stochastic) optimization problems including bundle methods and DCA~\cite{bolte2014kl,attouch2013convergence,le2018KL,liu2019pdcae,noll2014KL}. 
However, many of these results do not work explicitly with constrained problems, particularly for nonsmooth objectives.
In~\cite{atenas2023bundlelocal}, the authors showed that bundle methods applied to weakly convex objective and a convex feasible region  could generate linearly convergent serious steps if certain error bounds are satisfied.
In~\cite{yu2021KL}, the authors choose to enforce the constraints at the subproblem level by quadratic approximation, thus producing a feasible sequence of iterates. 
In many DCAs, nonsmooth objective is often relaxed from DC to weakly convex~\cite{yu2021KL}. 
In~\cite{liu2019pdcae}, the authors used a conjugate function to prove convergence of a version of DCA for nonconvex DC functions.  
In this paper, we study the local convergence behavior of our proposed SQP algorithm on nonsmooth problems with upper-$\Ctwo$ objectives, using the KL property assumptions. We make minimum changes to the globally convergent algorithm while focusing on reasonable assumptions that could ensure a transition to local convergence.

The paper is organized as follows. In Section~\ref{sec:prob}, we describe the mathematical background and notations.
In Section~\ref{sec:alg}, the line search SQP algorithm is proposed and its global convergence analysis  
is provided under reasonable assumptions.
In Section~\ref{sec:KLfunctions}, we demonstrate a group of optimization problems that have both an upper-$\Ctwo$ objective and a KL potential function.
The local convergence analysis based on KL property is presented in Section~\ref{sec:local}. 
Numerical experiments are shown in Section~\ref{sec:exp} where we apply the algorithm to a SCACOPF problem.

%% file: Sections_journal/Twostage.tex
\newcommand{\norm}[1]{\left\lVert {#1} \right\rVert}
\section{Background and notations}\label{sec:prob}

In this section, we provide the mathematical background and notations necessary for the analysis in this paper. 
We use $\norm{\cdot}$ to denote the standard Euclidean norm (2-norm), $\norm{\cdot}_1$ the 1-norm and $\norm{\cdot}_{\infty}$ the infinity norm or maximum norm. The inner product in $\Rbb^n$ is denoted as $\langle \cdot \rangle$.
The domain of an extended real-valued function $f:\Rbb^n\to \Rbb\cup\{+\infty\}$ is denoted as $\text{dom} f:= \{x: f(x)<\infty\}$. 
A function $f$ is proper if $\text{dom} f\neq \emptyset$. 
The graph of $f$ is defined as $G(f):=\{(x,u)\in\Rbb^n\times\Rbb: f(x)=u\}$. 

The lower regular subdifferential of a function $f:\Rbb^n\to\Rbb$ at point $\bar{x}$, 
denoted as $\hat{\partial} f(\bar{x})$, is defined by  
\begin{equation}\label{eqn:subgradient-def}
\centering
 \begin{aligned}
	\hat{\partial} f(\bar{x}) := \left\{ g\in\Rbb^n |\liminf_{ \substack{x\to \bar{x}\\x\neq \bar{x} }}\frac{f(x)-f(\bar{x})-\langle g,x-\bar{x}\rangle  }{\norm{x-\bar{x}}}\geq 0\right\}.
 \end{aligned}
\end{equation}
%The notion of \textit{f-attentive} convergence, which is crucial for the concept of general subgradient, is defined as 
%\begin{equation}\label{def:f-attentive}
%\centering
% \begin{aligned}
%	 x^{\nu} \xrightarrow[r]{} \bar{x} \quad \ratio\Leftrightarrow \quad x^{\nu}\to\bar{x} \quad \text{with} \quad r(x^{\nu})\to r(\bar{x}),
% \end{aligned}
%\end{equation}
%where $\{x^{\nu}\}$ is a sequence of points. Given the assumption of Lipschitz $r$, this becomes trivial.
%The lower general/limit subdifferential $\partial r(\bar{x})$ consists of all ordinary limits of sequences of lower regular subgradients $g^{\nu}$ at points $x^{\nu}$ that approach $\bar{x}$ in such a way that $r(x^{\nu})$ approaches $r(\bar{x})$.
For a Lipschitz $f(\cdot)$, if there exists a sequence $\{x^{\nu}\}$ such that $x^{\nu} \to \bar{x}$ and $g^{\nu}\in \hat{\partial}r(x^{\nu})$ with $g^{\nu}\to \bar{g}$, then 
$\bar{g}$ is a lower general subgradient (or lower limiting subgradient) of $f(\bar{x})$, written as $\bar{g} \in \partial f(\bar{x})$.

A Lipschitz function $f$ is lower regular (or subdifferentially regular) %7.25,~\cite{rockafellar1998}), 
if and only if $\partial f(\bar{x})=\hat{\partial} f(\bar{x})$~\cite[Corollary~8.11]{rockafellar1998}. 
Lower general subgradient is often simply called general subgradient, 
and a lower regular function is called regular. 
On the other hand, 
%A vector $v$ is called a proximal subgradient of a function $r:\Rbb^n\to \Rbb$ at $\bar{x}$, where $r(\bar{x})$ is finite, if 
%there exists $\rho>0$ and $\delta>0$ such that
%\begin{equation}\label{eqn:psub-def}
%\centering
% \begin{aligned}
%	 r(x) - r(\bar{x}) - \langle v,x-\bar{x}\rangle \geq -\rho \norm{x-x}^2
% \end{aligned}
%\end{equation}
%when $\norm{x-\bar{x}}<\delta$. A proximal subgradient, if it exists, belongs to a convex subset of the lower regular subgradient $\hat{\partial}r(\bar{x})$ (Prop. 8.46,~\cite{rockafellar1998}).
upper regular subdifferential~\cite{rockafellar1998,mordukhovich2004upp} is defined as
\begin{equation}\label{def:upp-subgradient}
\centering
 \begin{aligned}
	 \hat{\partial}^+ f(\bar{x}) :=& - \hat{\partial} (-f)(\bar{x})
	 =\left\{ g\in\Rbb^n |\limsup_{ \substack{x\to \bar{x}\\x\neq \bar{x} }}\frac{f(x)-f(\bar{x})-\langle g,x-\bar{x}\rangle  }{\norm{x-\bar{x}}}\leq 0\right\}.
 \end{aligned}
\end{equation}
Similarly, the upper general subdifferential is given by $\partial^+ f(\bar{x}) := -\partial (-f)(\bar{x})$.
A function $f$ is called upper regular if $-f$ is lower regular. 
Examples of upper regular functions include all continuous concave and continuously differentiable functions.

The Clarke subdifferential~\cite{clarke1983} is denoted as $\bar{\partial} f{(\bar{x})}$, which has also been widely adopted.
If $f$ is lower/upper regular, its Clarke subgradient and lower/upper general subgradient can be used interchangeably. We use Clarke subgradient for global convergence analysis and (lower) general subgradient in Section~\ref{sec:local} for a convex potential function.
An important property of these subdifferential is the outer/upper-semicontinuity, necessary in establishing convergence~\cite[Proposition~6.6]{rockafellar1998}. 
In addition, for a Lipschitz $r$, $\bar{\partial} f(\bar{x})$ is locally bounded~\cite[Theorem~9.13]{rockafellar1998}. 

A more restrictive but useful property than regularity is lower-$C^k$~\cite{Spingarn1981SubmonotoneSO,rockafellar1998,daniilidis2004}. 
A function $r:O\to \Rbb$, where $O\subset\Rbb^n$ is open, is said to be lower-$C^k$ on $O$, if on some neighborhood $V$ of each $\bar{x}\in O$ there is a representation
\begin{equation}\label{eqn:lowc1-def-0}
\centering
 \begin{aligned}
	 f(x) =  \underset{\substack{t \in T}}{\text{max}} \ f_t(x),
 \end{aligned}
\end{equation}
where $f_t:\Rbb^n\to\Rbb$ is of class $C^k$ on $V$ and the index set $T$ is a compact space such that $f_t$ 
and its partial derivatives in $x$ through order $k$ are jointly continuous on $(t,x)\in T\times V$. 
A function is called upper-$C^k$ if $-f$ is lower-$C^2$ or if we replace the $\text{max}$ with $\text{min}$ in~\eqref{eqn:lowc1-def-0}.
%Similarly, a function is upper-$C^k$ on $O$ if on a neighborhood $V$ of $\bar{x} \in O$ we can write
%\begin{equation}\label{eqn:uppc1-def-0}
%\centering
% \begin{aligned}
%	 r(x) =  \underset{\substack{t \in T}}{\text{min}} \ r_t(x),
% \end{aligned}
%\end{equation}
%where $r_t(\cdot)$ is of class $C^k$ on $V$, $T$ is compact and $r_t(x)$ 
%and its $k$th order partial derivatives are jointly continuous on $(t,x)\in T\times V$.
Let $T\subset\Rbb^p$ be compact, $f$ is upper-$C^2$ if it can be expressed as
\begin{equation}\label{eqn:uppc2-def-1}
\centering
 \begin{aligned}
	 f(x) =  \underset{\substack{t \in T}}{\text{min}} \ p(t,x)
 \end{aligned}
\end{equation}
for all $x\in O$, such that $p:\Rbb^p\times\Rbb^n\to\Rbb$ and its first- and second-order partial derivatives in $x$ 
depend continuously on $(t,x)$. Clearly, upper-/lower-$C^k$ imply upper-/lower-regularity.
From its definition~\eqref{eqn:uppc2-def-1}, two-stage (stochastic) optimization problems that are coupled only in a smooth objective 
have upper-$C^2$ objective in the first-stage problem, regardless of the complexity of the feasible region of second-stage problems. Constraint-coupled second-stage problems might also be upper-$C^2$~\cite{liu2020} or can be relaxed to obtain an upper-$C^2$ value function, using for example, quadratic penalty regularization~\cite{wang2022}. 

%Rockafellar~\cite{rockafellar1981Fav} and Clarke~\cite{clarke1995Proximal} further simplified the objective in~\eqref{eqn:uppc1-def-1} for Lipschitz functions. If $r:U\to\Rbb$, where $U\in\Rbb^n$ is an open, convex and bounded set, is Lipschitz, then it is upper-$C^2$ 
%if there exists $\sigma>0$, a compact set $S$ and continuous functions $b:S\to\Rbb^n, c:S\to\Rbb$ such that
%\begin{equation}\label{eqn:uppc1-def-lip}
%\centering
% \begin{aligned}
%	 r(x) =  \underset{\substack{s \in S}}{\text{min}} \{ \sigma \norm{x}^2 - \langle b(s),x\rangle -c(s) \}
% \end{aligned}
%\end{equation}
%for $\forall x\in U$. 

Equivalently, a finite-valued function $f$ is lower-$C^2$ on $O\subset\Rbb^n$ if and only if there exists $\rho>0$ such that 
$f(\cdot)+\frac{1}{2}\rho\norm{\cdot}^2$ is convex relative to some neighborhood of each point of $O$. %(10.33,~\cite{rockafellar1998}). 
For any compact subset $D\subset O$, one can find a uniform $\rho$ over $D$.
If $f$ is upper-$\Ctwo$ on $D$, there exists $\rho>0$ such that 
\begin{equation}\label{eqn:uppc2-convex}
\centering
 \begin{aligned}
       -f(\cdot) + \frac{1}{2} \rho\norm{\cdot}^2
 \end{aligned}
\end{equation}
is convex.
Consequently, 
for an upper-$\Ctwo$ $f(\cdot)$ we have  
\begin{equation}\label{eqn:uppc2-def}
\centering
 \begin{aligned}
	 f(x) - f(\bar{x}) - \langle g,x -\bar{x} \rangle \leq \frac{\rho}{2}\norm{x-\bar{x}}^2, 
 \end{aligned}
\end{equation}
where $g \in \bar{\partial} f(\bar{x})$ and $x\in D \subset O$~(8.12,~\cite{rockafellar1998}).
The domain of the (lower) subdifferential is defined as $\text{dom} \partial f:= \{x: \partial f(x)\neq \emptyset\}$. 
For an upper-$\Ctwo$ function at $x\in \text{dom} f$, $\bar{\partial} f(x)\neq \emptyset$,~\cite[Theorem~10.31]{clarke1983}. 

For problem~\eqref{eqn:opt0}, a necessary first-order optimality condition at a local minimum $\bar{x}$ is that there exists $\bar{\lambda} \in \Rbb^m $ with 
components $\bar{\lambda}_i, i\in \mathcal{E}\cup\mathcal{I}$ such that 
\begin{equation} \label{eqn:opt0-KKT}
   \centering
  \begin{aligned}
	  0 \in  \bar{\partial} f(\bar{x}) - \sum_{i\in \mathcal{E}\cup\mathcal{I}} \bar{\lambda}_i \nabla c_i (\bar{x}), \\
	  c_i(\bar{x}) = 0, \ i\in \mathcal{E}&, \\
	  \bar{\lambda}_i c_i(\bar{x}) = 0, c_i(\bar{x}) \geq 0, i\in \mathcal{I}&,\\
	  \bar{\lambda}_i \geq 0, i \in \mathcal{I}&.
  \end{aligned}
\end{equation}
A point that satisfies~\eqref{eqn:opt0-KKT} is called a KKT point of~\eqref{eqn:opt0}.
For upper regular functions, it is possible to establish a stronger form 
of subgradient optimality condition~\cite{mordukhovich2004upp}. 

For a nonempty set $C\subset \Rbb^n$, we denote the distance from a point $x\in\Rbb^n$ to $C$ as $\text{dist}(x,C)$, where
\begin{equation} \label{def:set-distance}
   \centering
  \begin{aligned}
    \mathrm{dist}(x,C) := \text{inf}\{\norm{y-x}:y\in C\}.  
  \end{aligned}
\end{equation}
The indicator function $i_C(\cdot)$ of a closed set $C\subset\Rbb^m$ is defined as
\begin{equation} \label{eqn:def-ind}
 \centering
  \begin{aligned}
          i_C(x)  =   \begin{cases}
                  0,  &\text{if} \ x \in C, \\ 
                  +\infty, &\text{otherwise.}
    \end{cases}
  \end{aligned}
\end{equation}

The conjugate function of a proper function $f$ is defined as
\begin{equation} \label{def:conjugate}
   \centering
  \begin{aligned}
     f^*(u) := \sup_{x} \{ \langle u,x\rangle  -f(x) \}.
  \end{aligned}
\end{equation}
From~\cite[Theorem~11.1]{rockafellar1998}, $f^*$ is proper, lower semicontinuous (lsc) and convex.
\begin{proposition}\label{prop:conj} 
If $f$ is proper, lsc and convex, then $\partial f^* = (\partial f)^{-1}$ and $\partial f=(\partial f^*)^{-1}$. Or
 \begin{equation} \label{def:conjugate-subgradient}
   \centering
  \begin{aligned}
     u \in \partial f(x)  \Leftrightarrow x \in \partial f^*(u)  \Leftrightarrow f(x)+f^*(u) = \langle u,x\rangle,
  \end{aligned}
\end{equation}
while $f(x)+f^*(u) \geq \langle u,x\rangle$ holds for all $x,u$~\cite[Proposition~11.3]{rockafellar1998}.
\end{proposition}

Next, we recall the KL property~\cite{bolte2014kl,attouch2010proximal}. 
\begin{definition}\label{def:KL}
The function $f$ is said to have the Kurdyka-Łojasiewicz (KL) property at $\hat{x}\in \mathrm{dom}\partial f$ if there exists $\eta\in (0,\infty]$, a neighborhood $U$ of $\hat{x}$ and a continuous concave function $\phi:[0,\eta)\to [0,\infty)$ such that\\
 \indent (\expandafter{\romannumeral 1}) $\varphi(0) = 0$,\\
 \indent (\expandafter{\romannumeral 2}) $\varphi$ is $C^1$ on $(0,\eta)$,\\
  \indent (\expandafter{\romannumeral 3}) $\varphi'>0$ on $(0,\eta)$;\\
   \indent (\expandafter{\romannumeral 4}) for all $x\in U$ and $f(\hat{x})< f(x) < f(\hat{x})+\eta$, the Kurdyka-Łojasiewicz inequality holds 
\begin{equation}
     \varphi'(f(x)-f(\hat{x})) \mathrm{dist}(0, \partial f(x)) \geq 1.
\end{equation}
 If $f(\cdot)$ satisfies the KL property at each point of $\text{dom} \partial f$, then $f(\cdot)$ is called a KL function.
\end{definition}
Moreover, if $\varphi(x) = a_0 x^{1-\alpha}$, $\alpha\in[0,1)$ and $f(\cdot)$ satisfies the KL property, then  $f(\cdot)$ is said to satisfy KL property with exponent $\alpha$.
We note that the definition~\ref{def:KL} uses general subdifferential $\partial f(x)$, which increases the difficulty of applying it to our problem because a general subgradient is not guaranteed to exist for an upper-$\Ctwo$ function.  

Another important property that we utilize is the subanalyticity of sets and functions, which has been studied extensively due to their prevalence in applications and relationship with KL property~\cite{bierstone1988semianalytic,bolte2007lojasiewicz,kurdyka1998gradients}. For example, a proper, subanalytic, continuous function on a closed domain is a KL function with exponent $\alpha$. We omit the definitions here for brevity of presentation and rely on existing propositions and theorems in our analysis. We only require subanalytic functions and sets in Section~\ref{sec:KLfunctions}.

%% file: Sections_journal/Algorithm2.tex
\section{\normalsize Nonsmooth SQP Algorithm}\label{sec:alg}
Motivated by the discussion in Section~\ref{se:intro} and~\ref{sec:prob}, we make the following assumptions. 
\begin{assumption}\label{assp:upperC2}
	The objective $f(\cdot)$ in~\eqref{eqn:opt0} is Lipschitz continuous and upper-$\Ctwo$. 
\end{assumption}
\noindent In particular, inequality~\eqref{eqn:uppc2-def} is satisfied.
Next, we formalize the smoothness of the constraints, a common assumption in literature~\cite{curtis2012}.
\begin{assumption}\label{assp:boundedHc}
	The functions $c_i(\cdot)$ are continuously differentiable with Lipschitz continuous gradient. Namely, there exists a constant $H\geq 0$ such that 
\begin{equation} \label{eqn:l-smooth-constraint}
 \centering
  \begin{aligned}
	  |c_i(x') - c_i(x) - \nabla c_i(x)^T (x'-x) | \leq&  \frac{H}{2} \norm{x-x'}^2, \\
  \end{aligned}
\end{equation}
	for all feasible $x,x'$ and $i\in \mathcal{E}\cup\mathcal{I}$~\cite{nesterovconvex2003}.
\end{assumption}

\subsection{Algorithm description}\label{sec:alg-1}
At iteration $k$, the objective $f$ is approximated with a local quadratic function $m_k(\cdot)$, defined as
\begin{equation} \label{eqn:opt-appx-x}
 \centering
  \begin{aligned}
	  m_k(x) = f(x_k) + g_k^T(x-x_k) + \frac{1}{2}(x-x_k)^T B_k (x-x_k),
  \end{aligned}
\end{equation}
where $g_k\in\bar{\partial} f(x_k)$, and $B_k$ is a bounded, symmetric, and positive definite matrix. 
Denoting $d=x-x_k$, $m_k(x)$ can be rewritten as $M_k(d)$: 
\begin{equation} \label{eqn:opt-rc-appx}
 \centering
  \begin{aligned}
	  M_k(d) &= f(x_k) + g_k^T d +\frac{1}{2} d^T B_k d.\\ 
  \end{aligned}
\end{equation}
The function value and subgradient at $x_k$ are exact, \textit{i.e.}, $M_k(0)=f(x_k),\nabla M_k(0)=g_k$.
The optimization subproblem is
\begin{equation} \label{eqn:opt-sqp-sub}
 \centering
  \begin{aligned}
   &\underset{\substack{d}}{\text{minimize}} 
	  & & M_k(d)\\
   &\text{subject to}
	  & & c_i(x_k) + \nabla c_i(x_k)^T d= 0, \ i\in\mathcal{E}\\ 
	  &&& c_i(x_k) +\nabla c_i(x_k)^T d \geq 0, \ i\in\mathcal{I}.
  \end{aligned}
\end{equation}
It is possible that the linearized constraints in~\eqref{eqn:opt-sqp-sub} are infeasible. 
Hence, similar to S$\ell_1$QP algorithm~\cite{fletcher2013}, we solve the following relaxed quadratic programming problem:
\begin{equation} \label{eqn:opt-sqp-penal}
 \centering
  \begin{aligned}
   &\underset{\substack{d,v,w,t}}{\text{min}} 
	  & & M_k(d) +\theta_k \sum_{i\in\mathcal{E}} (v_i + w_i) +\theta_k \sum_{i\in\mathcal{I}} t_i\\
   &\text{s.t.}
	  & & c_i(x_k)+\nabla c_i(x_k)^T d = v_i - w_i, \ i\in \mathcal{E},\\
	  &&& c_i(x_k)+\nabla c_i(x_k)^T d \geq -t_i, \ i\in\mathcal{I},\\
	  &&& v, w,t\geq 0,
  \end{aligned}
\end{equation}
where $v,w \in\Rbb^{m_e}$ and $t\in\Rbb^{m_i}$ are slack variables and $\theta_k>0$ is a penalty parameter. 
Denote the feasible region of~\eqref{eqn:opt-sqp-penal} as $\Omega_k$. We note that if the feasible region of~\eqref{eqn:opt-sqp-sub} 
is bounded for $x$, the slack variables in~\eqref{eqn:opt-sqp-penal} are effectively bounded as well. It is then reasonable to assume $\Omega_k$ to be bounded for $x,v,w,t$ given a bounded $\Omega$.

Let $d_k$, $v^k$, $w^k$ and $t^k$ denote the solutions to~\eqref{eqn:opt-sqp-penal}.
The first-order optimality conditions of problem~\eqref{eqn:opt-sqp-penal} for $d_k$ are
\begin{equation} \label{eqn:sqp-penal-KKT-1}
  \centering
   \begin{aligned}
	   g_k + B_k d_k - \sum_{i\in\mathcal{E}\cup \mathcal{I}} \lambda^{k+1}_i \nabla c_i(x_k) =0&,\\
	   c_i(x_k)+\nabla c_i(x_k)^T d_k - v^k_i + w^k_i =0, \ i\in\mathcal{E}&,\\
	   \lambda_i^{k+1} [c_i(x_k)+\nabla c_i(x_k)^T d_k + t^k_i] = 0, \ i\in\mathcal{I}&,\\
	   \lambda_i^{k+1}, c_i(x_k)+\nabla c_i(x_k)^T d_k + t^k_i\geq 0, \ i\in\mathcal{I}&.\\
   \end{aligned}
 \end{equation}
Here, $\lambda^{k+1}\in\Rbb^m$ is the Lagrange multiplier vector.
The remaining optimality conditions on slack variables $v^k$, $w^k$ and $t^k$ are 
\begin{equation} \label{eqn:sqp-penal-KKT-2}
  \centering
   \begin{aligned}
	   \theta_k + \lambda^{k+1}_i - p^{k+1}_i =0, \ \theta_k -\lambda^{k+1}_i-q^{k+1}_i=0, \ i\in \mathcal{E},\\
	   \theta_k - \lambda^{k+1}_i-r^{k+1}_i=0, \ i\in\mathcal{I},\\
	   p^{k+1}_i v^k_i  = 0, \ q^{k+1}_i w^k_i = 0, i\in \mathcal{E}, \ r_i^{k+1} t^k_i=0, \ i\in \mathcal{I}&,\\
	   v^k,w^k,t^k,p^{k+1},q^{k+1},r^{k+1} \geq 0&,\\
   \end{aligned}
 \end{equation}
where $p^{k+1}\in\Rbb^{m_e}$ and $q^{k+1}\in\Rbb^{m_e}$ are the Lagrange multipliers for bound constraints on $v^k$ and $w^k$, respectively, while $r^{k+1}\in\Rbb^{m_i}$ is the Lagrange multiplier for $t^k$.

To simplify notations, we define $[y]^-=\max\{0,-y\}$.
The constraint violation function is defined as 
\begin{equation} \label{eqn:sqp-cons-violation}
  \centering
    \begin{aligned}
	    v(x) = \sum_{i\in\mathcal{E}} |c_i(x)| + \sum_{i\in\mathcal{I}} [c_i(x)]^-.
    \end{aligned}
\end{equation}
Notice that $v(x)\geq 0$ and $v(x)=0$ only if $x\in\Omega$, .

To measure progress in both the objective and constraints, 
a $\ell_1$ merit function is adopted in the form of
\begin{equation} \label{eqn:opt-sqp-merit}
 \centering
  \begin{aligned}
	  \phi(x,\theta_k) = r(x) + \theta_k v(x). \\
  \end{aligned}
\end{equation}

Let the line search step size be $\alpha_k \in (0,1]$. 
The $k+1$ iterate is $x_{k+1} = x_k + \alpha_k d_k$.
We adopt the following classic line search criterion 
\begin{equation}\label{eqn:line-search-cond}
   \begin{aligned}
   \centering
       \phi(x_k,\theta_k) - \phi(x_{k+1},\theta_k) 
	   \geq \eta \alpha_k \frac{1}{2} d_k^T B_k d_k, 
   \end{aligned}
\end{equation}
where $\eta\in (0,1)$ is the line search coefficient.

To proceed, we define the sign function $\sigma_i^k:\Rbb^n\to\Rbb$, $i\in\mathcal{E}$  
based on whether the slack variable bound constraints are active:
\begin{equation} \label{eqn:opt-sqp-sign}
 \centering
  \begin{aligned}
	  \sigma_i^k(d)  =  
	  \begin{cases}
		  -1,  
		       & c_i(x_k)+\nabla c_i(x_k)^T d < 0, \\  
		  \phantom{-}0,
		  & c_i(x_k) + \nabla c_i(x_k)^T d = 0,\\
		  \phantom{-}1,
		  & c_i(x_k) + \nabla c_i(x_k)^T d > 0.
    \end{cases}
  \end{aligned}
\end{equation}
We define for $d_k$
%we divide the constraints into two sets based on the value of $ c_j(x_k)+ \nabla c_j(x_k)^T d_{k}$. 
%For simplicity, the two sets are referred to as the set of active and inactive equality constraints as we do without slack variables. 
the consistent linearized constraint set as
\begin{equation} \label{eqn:opt-sqp-fea-A}
   \centering
    \begin{aligned}
	    A_k = \{ i\in\mathcal{E}|v_i^k=w_i^k = 0\}\cup 
	    \{i\in\mathcal{I}|t_i^k = 0\},
     \end{aligned}
\end{equation}
and the inconsistent linearized constraint set as
\begin{equation} \label{eqn:opt-sqp-fea-I}
  \centering
    \begin{aligned}
	    V_k = \{ i\in\mathcal{E}| v_i^kw_i^k \neq 0\} \cup
	    \{i\in\mathcal{I}|t_i^k > 0\}.
    \end{aligned}
\end{equation}
Further, define $A_k^e = \mathcal{E}\cap A_k$, $A_k^i=\mathcal{I}\cap A_k$, as well as
$V_k^e = \mathcal{E}\cap V_k$, $V_k^i=\mathcal{I}\cap V_k$.
%		  %
Our nonsmooth line search SQP algorithm is presented in Algorithm~\ref{alg:sqp}.
\begin{algorithm}
% \DontPrintSemicolon
% \SetAlgoNoLine 
   \caption{Nonsmooth SQP with line search}\label{alg:sqp}
	\begin{algorithmic}[1]
	    \STATE{Initialize $x_0$, $B_0$, $\theta_0$, stopping error tolerance $\epsilon$, $\epsilon_c$, and $k=0$.
		Choose scalars $0<\eta < 1$, $\tau_{\alpha}\in (0,1)$ and $\gamma>0$. }
	\FOR{$k=0,1,2,...$}
	  \STATE{Evaluate the function value $f(x_k)$ and subgradient $g_k\in \bar{\partial}f(x_k)$.}
	  \STATE{Form the quadratic function $M_k$ in~\eqref{eqn:opt-rc-appx} and solve 
		subproblem~\eqref{eqn:opt-sqp-penal} to obtain $d_{k}$ and Lagrange multiplier $\lambda^{k+1}$.} 
	  \IF{$\norm{d_k}\leq \epsilon$ \text{and} $v(x_k)\leq \epsilon_c$}
	    \STATE{Stop the iteration and exit the algorithm.}
	  \ENDIF
	   \STATE{Set $\theta_{k+1} = \max{\{\theta_{k}, \norm{\lambda^{k+1}}_{\infty}+\gamma\}}$. }
	   \STATE{Find the line search step size $\alpha_k>0$ using backtracking, starting at $\alpha_k=1$ and reducing $\alpha_k$ by 
		$\alpha_k = \tau_{\alpha} \alpha_k$ until the conditions in~\eqref{eqn:line-search-cond} are satisfied.}
           \STATE{Take the step $x_{k+1} = x_k+\alpha_k d_k$.}
	    \STATE{Call the chosen $B_k$ update rules to obtain $B_{k+1}$.}
       \ENDFOR
    \end{algorithmic}
\end{algorithm}

%%%%%%%%%%%%%%%%%%%%%%%%%%%%%%%%%%%%%%%%%%%%%%%%%%%%%%%%%%%%%
%%%%  Convergence
%%%%%%%%%%%%%%%%%%%%%%%%%%%%%%%%%%%%%%%%%%%%%%%%%%%%%%%%%%%%%

\subsection{Global convergence analysis}\label{sec:alg-convg}
Following discussions in Section~\ref{se:intro} and the authors' previous paper~\cite{wang2022}, 
we assume a bounded feasible region $\Omega_k$ throughout the iterations, instead of assumptions such as level-boundedness of $f(\cdot)$. 
We note again that a bounded $\Omega_k$ can be easily achieved if $x$ is bounded above and below, which is the case in many engineering applications, \textit{e.g}, power grid optimization.  
Since $B_k$ no longer approximates a Hessian, it can be required to be positive definite. The iterative assumptions are summarized below. 
\begin{assumption}\label{assp:Bk}
    The feasible region $\Omega$ and $\Omega_k$ are bounded for all $k$.  
     The matrix $B_k$ is bounded and positive-definite. Thus, there exists $b>0$ such that 
        $d^T B_k d \geq b \norm{d}^2$ for all $k$ and $d\in\Rbb^n$.
\end{assumption}
Assumptions~\ref{assp:upperC2},~\ref{assp:boundedHc} and~\ref{assp:Bk} are assumed valid throughout the analysis.
If the algorithm terminates in a finite number of steps, the stopping test at step 4 is satisfied with the error tolerance $\epsilon,\epsilon_c$. 
Let $\epsilon=\epsilon_c=0$, then $\norm{d_k}= 0$ and $v(x_k)=0$.
Since $d_k$ solves~\eqref{eqn:opt-sqp-penal}, optimality conditions in~\eqref{eqn:sqp-penal-KKT-1}
are satisfied, of which the first equation reduces to 
\begin{equation}\label{alg:finite-steps}
 \begin{aligned}
	 g_k - \sum_{i\in \mathcal{E}\cup\mathcal{I}} \lambda^{k+1}_i \nabla c_i(x_k) =0.
 \end{aligned}
\end{equation}
Given $g_k\in\bar{\partial} f(x_k)$, we have $0\in \bar{\partial} f(x_k) - \sum_{i\in\mathcal{E}\cup\mathcal{I}} \lambda^{k+1}_i\nabla c(x_k) $. 
In addition,  $x_k$ is feasible.
From the remaining equations in~\eqref{eqn:sqp-penal-KKT-1} 
and~\eqref{eqn:sqp-penal-KKT-2}, $x_k$ satisfies~\eqref{eqn:opt0-KKT} and is by definition a KKT point for~\eqref{eqn:opt0} as the algorithm exits.

In what follows, the analysis is focused on the case with an infinite number of steps, \textit{i.e.}, $\norm{d_k}>0$ or $v(x_k) >0$. 
We start with the following lemma that establishes the relationship between $\lambda^{k+1}$ and $\theta_k$.
%%%%%%%%%%%%%%%%%%%%%%%%%%%%%%%%%%%%%%%%%%%%%%%%%%%%%%%%%%%%%%%%%%%%%%
\begin{lemma}\label{lem:lambda-sqp-prop}
	If $i\in V_k^e$, then $\lambda_i^{k+1} = -\sigma_i^k(d_k)\theta_k$. 
	If $i\in A_k^e$, then $\lambda_i^{k+1}\in [-\theta_k,\theta_k]$.
	Similarly, if $i\in V_k^i$, then $\lambda_i^{k+1} = \theta_k$.
	If $i\in A_k^i$, then $\lambda_i^{k+1}\in [0,\theta_k)$.
\end{lemma}
\begin{proof}
	Note first that for any $ i\in\mathcal{E}$, the slack variable solutions satisfy $v^k_i w^k_i = 0$. 
	We consider the three cases given by the value of $\sigma_i^k(d_k),i\in\mathcal{E}$.
	If $\sigma_i^k(d_k)=1$, by the second equation in~\eqref{eqn:sqp-penal-KKT-1}, $v_i^k>0$ and $w_i^k=0$.
	From the third line in~\eqref{eqn:sqp-penal-KKT-2}, 
        $p^{k+1}_i=0$. Then, by the first equation in~\eqref{eqn:sqp-penal-KKT-2}, $\lambda_i^{k+1}=-\theta_k$.
	Similarly, if $\sigma_i^k(d_k)=-1$, 
        one obtains $q^{k+1}_i=0$ and $\lambda_i^{k+1}=\theta_k$. The first part of the Lemma is proven.

	If $i\in  A_k^e$, then $\sigma_i^k(d_k)=0$ by the second equation in~\eqref{eqn:sqp-penal-KKT-1}.
We sum and subtract the equations on the first line in~\eqref{eqn:sqp-penal-KKT-2} to obtain 
%      \begin{equation} \label{eqn:lambda-pq}
%       \centering
%        \begin{aligned}
		$\lambda_i^{k+1} = \frac{1}{2}(p_i^{k+1}-q_i^{k+1})$ and %\\
		$p_i^{k+1}+q_i^{k+1}=2\theta_k$. %\\
        %\end{aligned}
       %\end{equation}
	Since $p^{k+1}\geq 0$ and $q^{k+1}\geq 0$, the latter equation implies $0\leq p_i^{k+1}$ and $q_i^{k+1}\leq 2\theta_k$; finally,
	the former equation implies $\lambda_i^{k+1}\in [-\theta_k,\theta_k]$.

	If $i\in  V_k^i$, then $r_i^{k+1} = 0$ from the third line in~\eqref{eqn:sqp-penal-KKT-2}. Thus, 
	$\lambda_i^{k+1} = \theta_k$ from the second equation in~\eqref{eqn:sqp-penal-KKT-2}.
	If $i\in A_k^i$, then $r_i^{k+1}>0$ and by the second line of~\eqref{eqn:sqp-penal-KKT-2} $\lambda_i^{k+1} \in [0 ,\theta_k)$.
\end{proof}

\begin{remark}\label{rmrk:realalpha}
	We do not specify how $B_k$ is updated in this paper. Given the nonsmooth nature of $f(\cdot)$, 
        an approximation to Hessian is not possible. 
	We do encourage a diagonal matrix so that the sparsity structure of the linear system is preserved~\cite{wang2022}.
\end{remark}

%%%%%%%%%%%%%%%%%%%%%%%%%%%%%%%%%%%%%%%%%%%%%%%%
\begin{lemma}\label{lem:line-search-eqcons}
	The equality constraints satisfy the following inequality 
       \begin{equation} \label{eqn:line-search-cons-eq}
         \centering
         \begin{aligned}
		 \theta_k \sum_{i\in\mathcal{E}}|c_i(x_k)| - \theta_k \sum_{i\in\mathcal{E}}|c_i(x_{k+1})| - \alpha_k \sum_{i\in\mathcal{E}} \lambda^{k+1}_i\nabla c_i(x_k)^T d_k
                     \geq -  \frac{1}{2}\theta_k \alpha_k^2 c_h^e \norm{d_k}^2,
      \end{aligned}
     \end{equation}
	where $c_h^e>0$ is a constant coefficient.
\end{lemma}
\begin{proof}
	By Assumption~\ref{assp:boundedHc}, for $x_{k+1}=x_k+\alpha_k d_k$,
       \begin{equation} \label{eqn:sqp-c-ls-pf-1}
         \centering
         \begin{aligned}
		 |c_i(x_{k+1})  - c_i(x_k)- \alpha_k \nabla c_i(x_k)^T d_{k}| \leq \frac{1}{2}\alpha_k^2 H \norm{d_k}^2.\\
      \end{aligned}
     \end{equation}
	By triangle inequality, we have
   \begin{equation} \label{eqn:sqp-c-ls-pf-2}
   \centering
    \begin{aligned}
	    |c_i(x_{k+1})|  
	    &\leq | c_i(x_k) + \alpha_k \nabla c_i(x_k)^T d_{k}| + \frac{1}{2}\alpha_k^2 H \norm{d_k}^2\\
	     &= \left|(1-\alpha_k) c_i(x_k) + \alpha_k \left[c_i(x_k)+ \nabla c_i(x_k)^T d_{k}\right]\right|+\frac{1}{2}\alpha_k^2 H \norm{d_k}^2\\ 
	      & \leq (1-\alpha_k)| c_i(x_k)| + \alpha_k |c_i(x_k)+ \nabla c_i(x_k)^T d_{k}|+ \frac{1}{2}\alpha_k^2 H \norm{d_k}^2 . 
    \end{aligned}
   \end{equation}
	Applying the definition~\eqref{eqn:opt-sqp-sign} of $\sigma_j^k$ to~\eqref{eqn:sqp-c-ls-pf-2}, we can write 
\begin{equation} \label{eqn:sqp-c-ls-pf-2.3}
   \centering
    \begin{aligned}
	    |c_i(x_{k+1})|  
		  & \leq (1-\alpha_k)|c_i(x_k)| + \alpha_k \sigma_i^k(d_k)\left[c_i(x_k)+ \nabla c_i(x_k)^T d_{k}\right] + \frac{1}{2}\alpha_k^2 H\norm{d_k}^2.
    \end{aligned}
   \end{equation}
	Let us denote the cardinality of $A_k^e$ and $V_k^e$ by $m_e^a$ and $m_e^v$, respectively. By summing up~\eqref{eqn:sqp-c-ls-pf-2.3} over $i\in V_k^e$ and noting that $|c_i(x_k)|\geq \sigma_i^k(d_k) c_i(x_k)$, we have
   \begin{equation} \label{eqn:sqp-c-ls-pf-2.5}
     \centering
      \begin{aligned}
	      \sum_{i\in  V_k^e}|c_i(x_{k+1})|  \leq& 
	       \sum_{i \in V_k^e}  |c_i(x_k)|+\alpha_k\sum_{i \in  V_k^e} \sigma_i^k(d_k)\nabla c_i(x_k)^Td_k 
	      + \frac{1}{2}m_e^v\alpha_k^2 H \norm{d_k}^2.\\
      \end{aligned}
    \end{equation}
	Similarly, we sum up~\eqref{eqn:sqp-c-ls-pf-2.3} over $i\in A_k^e$ and use the definition in~\eqref{eqn:opt-sqp-fea-A} to write 
   \begin{equation} \label{eqn:sqp-c-ls-pf-2.6}
     \centering
      \begin{aligned}
\sum_{i \in A_k^e}|c_i(x_{k+1})|  \leq& 
	      (1-\alpha_k)\sum_{i \in A_k^e}  |c_i(x_k)| + \frac{1}{2}m_e^a \alpha_k^2 H \norm{d_k}^2. \\
      \end{aligned}
    \end{equation}
    Summing the two equations in~\eqref{eqn:sqp-c-ls-pf-2.5} and~\eqref{eqn:sqp-c-ls-pf-2.6}
    and using $m_e^a+m_e^v=m_e$ gives us 
\begin{equation} \label{eqn:sqp-c-ls-pf-3}
     \centering
      \begin{aligned}
	      \sum_{i\in\mathcal{E}} |c_i(x_k)|-\sum_{i\in\mathcal{E}} |c_i(x_{k+1})|\geq  
	    \alpha_k  \sum_{i \in A_k^e}  |c_i(x_k)|   -\alpha_k \sum_{i \in V_k^e} \sigma_i^k(d_k)\nabla c_i(x_k)^T d_k 
	      - \frac{m_e}{2}\alpha_k^2 H \norm{d_k}^2. \\
      \end{aligned}
    \end{equation}
    From Lemma~\ref{lem:lambda-sqp-prop} and the definition~\eqref{eqn:opt-sqp-fea-A} of $A_k$, we can write
\begin{equation} \label{eqn:sqp-c-ls-pf-3.1}
     \centering
      \begin{aligned}
	      \sum_{i\in \mathcal{E}} \lambda_i^{k+1}\nabla c_i(x_k)^T d_k =& \sum_{i\in V_k^e} \lambda_i^{k+1} \nabla c_i(x_k)^T d_k +\sum_{i\in A_k^e} \lambda_i^{k+1} \nabla c_i(x_k)^T d_k\\
	      =& -\sum_{i\in V_k^e} \sigma_i^k(d_k)\theta_k \nabla c_i(x_k)^T d_k -\sum_{j\in A_k^e} \lambda^{k+1}_i  c_i(x_k).\\
	      \leq& -\sum_{i\in V_k^e} \sigma_i^k(d_k)\theta_k \nabla c_i(x_k)^T d_k + \theta_k\sum_{i\in A_k^e}  |c_i(x_k)|.\\
      \end{aligned}
    \end{equation}
    The inequality in~\eqref{eqn:sqp-c-ls-pf-3.1} comes from the second part of Lemma~\ref{lem:lambda-sqp-prop}, 
    namely $\lambda^{k+1}_i\in [-\theta_k,\theta_k]$ for $i\in A_k^e$.
	Through basic algebraic calculations and applying~\eqref{eqn:sqp-c-ls-pf-3.1} to~\eqref{eqn:sqp-c-ls-pf-3}, 
   \begin{equation} \label{eqn:sqp-c-ls-pf-3.5}
   \centering
    \begin{aligned}
	    &  \theta_k \sum_{i\in\mathcal{E}} |c_i(x_k)|-\theta_k\sum_{i\in\mathcal{E}} |c_i(x_{k+1})| - \alpha_k \sum_{i\in\mathcal{E}}\lambda^{k+1}_i\nabla c_i(x_k)^T d_k
	     \geq \\
	    &\qquad \theta_k \alpha_k \sum_{i \in A_k^e}  |c_i(x_k)|  
	      -\theta_k \alpha_k  \sum_{i\in V_k^e} \sigma_i^k(d_k)\nabla c_i(x_k)^T d_k
	      - \frac{1}{2} \theta_k m_e\alpha_k^2 H \norm{d_k}^2 \\ 
	    & \qquad+ \alpha_k\theta_k\sum_{i\in V_k^e} \sigma_i^k(d_k) \nabla c_i(x_k)^T d_k
	     - \alpha_k \theta_k\sum_{i\in A_k^e} | c_i(x_k) | = 
	    -  \frac{1}{2}\theta_k m_e \alpha_k^2 H \norm{d_k}^2.
    \end{aligned}
   \end{equation}
    Let $c_h^e = m_eH$ to complete the proof.
\end{proof}

\begin{lemma}\label{lem:line-search-ineqcons}
	The inequality constraints satisfy the following inequality 
       \begin{equation} \label{eqn:line-search-cons-ineq}
         \centering
         \begin{aligned}
		 \theta_k \sum_{i\in\mathcal{I}} [c_i(x_k)]^{-} - \theta_k \sum_{i\in\mathcal{I}} [c_i(x_{k+1})]^{-} 
		  - \alpha_k \sum_{i\in\mathcal{I}} \lambda^{k+1}_i\nabla c_i(x_k)^T d_k
                     \geq - \frac{1}{2}\alpha_k^2 \theta_k c_h^i \norm{d_k}^2,
      \end{aligned}
     \end{equation}
	where $c_h^i>0$ is a constant coefficient.
\end{lemma}
\begin{proof}
By Assumption~\ref{assp:boundedHc}, for $x_{k+1}=x_k+\alpha_k d_k$,
       \begin{equation} \label{eqn:sqp-ineq-ls-pf-1}
         \centering
         \begin{aligned}
		 \relax [c_i(x_{k+1})]^- \leq  [c_i(x_k)+ \alpha_k \nabla c_i(x_k)^T d_{k}]^- + \frac{1}{2}\alpha_k^2 H \norm{d_k}^2.\\
      \end{aligned}
     \end{equation}
	Multiplying by $-1$ and adding $[c_i(x_k)]^-$ to~\eqref{eqn:sqp-ineq-ls-pf-1} gives us 
       \begin{equation} \label{eqn:sqp-ineq-ls-pf-1.5}
         \centering
         \begin{aligned}
         \relax [c_i(x_{k})]^- - [c_i(x_{k+1})]^- \geq [c_i(x_{k})]^--[c_i(x_k)+ \alpha_k \nabla c_i(x_k)^T d_{k}]^- - \frac{1}{2}\alpha_k^2 H \norm{d_k}^2.\\
      \end{aligned}
     \end{equation}

	If $i\in A_k^i$, there exists two cases from the definition in~\eqref{eqn:opt-sqp-fea-A}. 
	In the first case, $c_i(x_k)+ \nabla c_i(x_k)^T d_{k} =0 $. 
	From~\eqref{eqn:sqp-ineq-ls-pf-1.5},
\begin{equation} \label{eqn:sqp-ineq-ls-pf-3}
         \centering
         \begin{aligned}
		 \relax [c_i(x_{k})]^- - [c_i(x_{k+1})]^-  \geq&   \alpha_k [c_i(x_{k})]^-  - \frac{1}{2}\alpha_k^2 H \norm{d_k}^2\\
		 \geq& \alpha_k [-\nabla c_i(x_k)^T d_k]^-  - \frac{1}{2}\alpha_k^2 H \norm{d_k}^2\\
		 \geq& \alpha_k\max\{0, \nabla c_i(x_k)^T d_k\}- \frac{1}{2}\alpha_k^2 H \norm{d_k}^2.
      \end{aligned}
     \end{equation}
	Applying Lemma~\ref{lem:lambda-sqp-prop} when $i\in A_k^i$ to~\eqref{eqn:sqp-ineq-ls-pf-3},
\begin{equation} \label{eqn:sqp-ineq-ls-pf-3.5}
         \centering
         \begin{aligned}
		 \theta_k [c_i(x_{k})]^- -\theta_k [c_i(x_{k+1})]^- -  \alpha_k \lambda_i^{k+1} \nabla c_i(x_k)^T d_k
		 \geq - \frac{1}{2}\alpha_k^2 \theta_k H \norm{d_k}^2.\\
      \end{aligned}
     \end{equation}
	In the second case, 
	$c_i(x_k)+ \nabla c_i(x_k)^T d_{k} >0 $, which leads to $\lambda_i^{k+1}=0$ through the complementarity condition. 
        There are multiple scenarios in regards to the sign of $c_i(x_k)$ and $\nabla c_i(x_k)^T d_{k}$. However, it can be easily verified that in all those scenarios, based on~\eqref{eqn:sqp-ineq-ls-pf-1.5}, we have 
	\begin{equation} \label{eqn:sqp-ineq-ls-pf-4}
         \centering
         \begin{aligned}
	\relax [c_i(x_{k})]^- - [c_i(x_{k+1})]^- \geq& - \frac{1}{2}\alpha_k^2 H \norm{d_k}^2.\\
      \end{aligned}
     \end{equation}
	Given that $\lambda_i^{k+1}=0$ in this case, we can multiply both sides of~\eqref{eqn:sqp-ineq-ls-pf-4} by $\theta_k$ to obtain~\eqref{eqn:sqp-ineq-ls-pf-3.5}. Hence,~\eqref{eqn:sqp-ineq-ls-pf-3.5} is satisfied for all $i\in A_k^i$.

	Next, if $i\in V_k^i$, then $c_i(x_k)+ \nabla c_i(x_k)^T d_{k} =t^k_i >0$ and $\lambda^{k+1}_i=0$. The analysis is similar to that of the second case of $i\in A_k^i$ above. And we can come to the same conclusion that~\eqref{eqn:sqp-ineq-ls-pf-3.5} is satisfied for all $i\in V_k^i$.
       Summing all $i\in \mathcal{I}$ and let $c_h^i=m_i H$ completes the proof.
\end{proof}

\begin{lemma}\label{lem:line-search-merit}
	If the Lagrange multipliers $\lambda^{k+1}$ are bounded, then step 8 of Algorithm~\ref{alg:sqp} finds $\alpha_k\in (0,1]$ satisfying the line search condition~\eqref{eqn:line-search-cond} in a finite number of steps.
\end{lemma}

\begin{proof}
      By Assumption~\ref{assp:upperC2}, we have  
	\begin{equation} \label{eqn:sqp-penal-merit-pf-0}
 \centering
  \begin{aligned}
	  f(x_k) - f(x_{k+1}) \geq& -\alpha_k g_{k}^Td_k -\frac{\rho}{2}  \alpha_k^2 \norm{d_k}^2. \\
  \end{aligned}
\end{equation}
	Let us rearrange the first equation in the KKT conditions~\eqref{eqn:sqp-penal-KKT-1} to obtain 
\begin{equation} \label{eqn:sqp-penal-merit-pf-1}
  \centering
   \begin{aligned}
	   g_k + B_k d_k  = \sum_{i\in \mathcal{E}\cup\mathcal{I}} \lambda_i^{k+1} \nabla c_i(x_k).\\
   \end{aligned}
 \end{equation}
	Taking the dot product with $-d_k$ on both sides of~\eqref{eqn:sqp-penal-merit-pf-1} leads to
\begin{equation} \label{eqn:sqp-penal-merit-pf-2}
  \centering
   \begin{aligned}
	   - g_k^Td_k - d_k^T B_k d_k =& -\sum_{i\in \mathcal{E}\cup\mathcal{I}} \lambda^{k+1}_i \nabla c_i(x_k)^T d_k.\\
   \end{aligned}
 \end{equation}
	Since $\alpha_k\in (0,1]$, one can multiply~\eqref{eqn:sqp-penal-merit-pf-2} by $\alpha_k$ and
	apply it to~\eqref{eqn:sqp-penal-merit-pf-0} to obtain
\begin{equation} \label{eqn:sqp-penal-merit-pf-3}
  \centering
   \begin{aligned}
	   f(x_k) - f(x_{k+1}) \geq \alpha_k d_k^T B_k d_k -\frac{\rho}{2}\alpha_k^2\norm{d_k}^2 - \alpha_k \sum_{i\in \mathcal{E}\cup\mathcal{I}}\lambda_i^{k+1} \nabla c_j(x_k)^T d_k
   \end{aligned}
 \end{equation}
	From Lemma~\ref{lem:line-search-eqcons},~\ref{lem:line-search-ineqcons}, Assumption~\ref{assp:Bk} and~\eqref{eqn:sqp-penal-merit-pf-3}, the merit function satisfies
\begin{equation*} \label{eqn:sqp-penal-merit-pf-4} 
 \centering
  \begin{aligned}
	  \phi&(x_k,\theta_k) - \phi(x_{k+1},\theta_k)  = f(x_k) -f(x_{k+1})+ \theta_k v(x_k)-\theta_kv(x_{k+1})\\
			 \geq & \alpha_k d_k^TB_kd_k - \frac{\rho}{2} \alpha_k^2\norm{d_k}^2- \alpha_k \sum_{i\in \mathcal{E}\cup\mathcal{I}} \lambda_i^{k+1} \nabla c_i(x_k)^T d_k + \theta_k v(x_k)	-\theta_k v(x_{k+1})\\
			\geq& \frac{1}{2}\alpha_k d_k^TB_kd_k + \frac{\alpha_k}{2} (b- \rho \alpha_k - \theta_k \alpha_k c_h)  \norm{d_k}^2, \\
  \end{aligned}
\end{equation*}
where $c_h=c_h^i+c_h^e$.
 Since $\lambda^k$ is bounded, 
 there exists $k$ such that $\theta_t=\theta_k>0$ for all $t\geq k$.
 Then step 8 line search completes successfully when $b>(\rho+\theta_k c_h)\alpha_k$ is achieved by reducing $\alpha_k$. 
\end{proof}

In order for the Lagrange multipliers to be bounded, a constraint qualification is necessary.
Given the existence of both equality and inequality constraints, we resort to MFCQ~\cite{Nocedal_book}.
Denote by $\mathcal{A}(x) = \mathcal{E}\cup \{i\in\mathcal{I}|c_i(x) = 0\}$ the active set at a feasible $x$. The MFCQ is defined below.
\begin{definition}\label{assp:MFCQ}
	The constraints in~\eqref{eqn:opt0} satisfy MFCQ at $x$ if the gradients $\nabla c_i(x)$, $i \in\mathcal{E}$ are linearly independent and there exists a vector $w\in \Rbb^n$ such that 
  \begin{equation}\label{def:mfcq}
   \centering
    \begin{aligned}
	\nabla c_i(x)^T w >0,  \ \text{for all} \ i \in \mathcal{A}(x)\cap \mathcal{I}, \\ 
        \nabla c_i(x)^T w = 0, \ \text{for all} \ i \in \mathcal{E}.
    \end{aligned}
  \end{equation}
\end{definition}
%%%%%%%%%%%%%%%%%%%%%%%%%%%%%%%%%%%%%%%%%%%%%
\begin{lemma}\label{lem:bounded-lp}
	If MFCQ of the constraints in~\eqref{eqn:opt0} are satisfied at all accumulation points of $\{x_k\}$ generated by the algorithm, then 
	the sequence of Lagrange multipliers $\{\lambda^{k+1}\}$ are bounded. 
	Additionally, there exists $k$ such that $\theta_t = \theta_k$ for all $t\geq k$.
\end{lemma}
\begin{proof}
     Suppose on the contrary, at least one of the Lagrange multipliers is unbounded. 
From Assumption~\ref{assp:Bk}, the sequences $\{x_k\}$ and $\{B_k\}$ are bounded. Further,$\{g_k\}$ is bounded due to Assumption~\ref{assp:upperC2}.
     Let $x_{k_u}\to\bar{x}$ as $k_u\to\infty$, where $\bar{x}$ is an accumulation point. Then, $\norm{\lambda^{k_{u+1}}}_{\infty}\to\infty$.  The corresponding subsequences $\{x_{k_u}\}$, $\{g_{k_u}\}$ remain bounded. 
     Rewrite the first equation in optimality condition in~\eqref{eqn:sqp-penal-KKT-1} as
    \begin{equation} \label{eqn:sqp-penal-KKT-full}
     \centering
     \begin{aligned}
	     g_k + B_k d_k = \sum_{i\in\mathcal{E}} \lambda^{k+1}_i \nabla c_i(x_k)+ \sum_{i\in\mathcal{I}\cap \mathcal{A}(\bar{x})} \lambda^{k+1}_i \nabla c_i(x_k) + \sum_{i\in\mathcal{I} \backslash \mathcal{A}(\bar{x})} \lambda^{k+1}_i \nabla c_i(x_k) ,\\
     \end{aligned}
    \end{equation}
	 The left-hand side of the equation stays bounded as $k_u\to\infty$. 

	Consider the case where there exists $i \in \mathcal{E}$ so that $\lambda_i^{k_{u+1}} \to\infty$.
        Let $t_k = \norm{\lambda^{k+1}_e}_{\infty}$, where $\lambda^{k+1}_e$ is the Lagrange multiplier vector for all $i\in\mathcal{E}$. Then, $t_{k_u}\to\infty$ as $k_u\to \infty$.
       Divide both sides of~\eqref{eqn:sqp-penal-KKT-full} by $t_{k_u}$, we have
     \begin{equation} \label{eqn:sqp-penal-KKT-full-2}
       \centering
       \begin{aligned}
	       \frac{1}{t_{k_u}} ( g_{k_u} + B_{k_u} d_{k_u} ) = 
	       \sum_{i\in\mathcal{E}}  a^{k_{u+1}}_i\nabla c_i(x_{k_u}) 
	       +\sum_{i\in\mathcal{I} \cap \mathcal{A}(\bar{x})}  b^{k_{u+1}}_i\nabla c_i(x_{k_u}) 
               +\sum_{i\in\mathcal{I} \backslash \mathcal{A}(\bar{x})} c^{k_{u+1}}_i \nabla c_i(x_k), 
       \end{aligned}
      \end{equation}
    where  
    \begin{equation} \label{eqn:sqp-penal-multi}
       \centering
       \begin{aligned}
	       a^{k_{u+1}}_i = \frac{\lambda^{k_u+1}_i}{ t_{k_u}}, \ i\in \mathcal{E}, \
              b^{k_{u+1}}_i =  \frac{\lambda^{k_u+1}_i}{ t_{k_u}},\ i\in \mathcal{I} \cap \mathcal{A}(\bar{x}), \ 
              c^{k_{u+1}}_i = \frac{\lambda^{k_u+1}_i}{ t_{k_u}},\ i\in \mathcal{I} \backslash \mathcal{A}(\bar{x}).
       \end{aligned}
      \end{equation}
Next, we notice that for $i\in\mathcal{I} \backslash \mathcal{A}(\bar{x})$,  $c(\bar{x})>0$. Hence, 
      the Lagrange multiplier $\bar{\lambda}_i = 0,i\in\mathcal{I} \backslash \mathcal{A}(\bar{x})$ at $\bar{x}$  due to the complementarity conditions.
      Thus, $c_i^{k_{u+1}}\to 0$. Taking the limit of~\eqref{eqn:sqp-penal-KKT-full-2}, we get 
     \begin{equation} \label{eqn:sqp-penal-KKT-full-3}
       \centering
       \begin{aligned}
        \lim_{k_u\to\infty} \sum_{i\in\mathcal{E}}  a^{k_{u+1}}_i\nabla c_i(x_{k_u}) 
	       +\sum_{i\in\mathcal{I} \cap \mathcal{A}(\bar{x})}  b^{k_{u+1}}_i\nabla c_i(x_{k_u}) =0.
        \end{aligned}
      \end{equation}
        Notice that there exists $j\in\mathcal{E}$ so that $a_j^{k_{u+1}} = \pm 1$ for all $k_u$.
Then, due to the linear independence of $\nabla c_i(\bar{x}), i\in\mathcal{E}$, the first term of~\eqref{eqn:sqp-penal-KKT-full-3} does not converge to zero. Consequently, there exists $j\in\mathcal{I}\cap \mathcal{A}(\bar{x})$ so that $b_j^{k_{u+1}} > M$ for all $k_u$ and $M>0$.

	From MFCQ at $\bar{x}$, there exists $w \in \Rbb^n$ such that~\eqref{def:mfcq} is satisfied at $\bar{x}$. Taking the dot product with $w$ of~\eqref{eqn:sqp-penal-KKT-full-3}, we have
     \begin{equation} \label{eqn:sqp-penal-KKT-full-5}
       \centering
       \begin{aligned}
      \lim_{k_u\to\infty}\sum_{i\in\mathcal{I} \cap \mathcal{A}(\bar{x})}  b^{k_{u+1}}_i\nabla c_i(x_{k_u})^T w =  -\lim_{k_u\to\infty} \sum_{i\in\mathcal{E}}  a^{k_{u+1}}_i\nabla c_i(x_{k_u})^T w =0.
       \end{aligned}
      \end{equation}
      However, we also know that for some $j\in\mathcal{I}\cap \mathcal{A}(\bar{x})$, $b_j^{k_{u+1}} > M$. Thus,
     \begin{equation} \label{eqn:sqp-penal-KKT-full-6}
       \centering
       \begin{aligned}
 \sum_{i\in\mathcal{I} \cap \mathcal{A}(\bar{x})}  b^{k_{u+1}}_j\nabla c_j(x_{k_u})^T w 
\geq  M  c_j(\bar{x})^T w > 0.
       \end{aligned}
      \end{equation}
       This is a contradiction. Therefore, $\{t_k\}$ is bounded. 
       There remains a case where only the Lagrange multipliers for the inequality constraints are unbounded. The proof is similar to the one above by setting $t_k=\norm{\lambda^{k+1}_i}$, where $\lambda^{k+1}_i$ is the vector for inequality Lagrange multipliers. The proof is hence omitted.
       We conclude that $\lambda^{k+1}$ is bounded.
	Since $\theta_k$ is determined by $\lambda^k$, there exists $k$ such that $\theta_t=\theta_k$ for all $t\geq k$. 
\end{proof}

\begin{lemma}\label{lem:merit-decrease}
       If the MFCQ conditions in Lemma~\ref{lem:bounded-lp} hold,
	then there exists $c_{\phi}>0$ so that the merit function sequence $\{\phi(x_k,\theta_k)\}$ satisfies  
\begin{equation*} \label{eqn:merit-decrease-1} 
 \centering
  \begin{aligned}
	  \phi&(x_k,\theta_k) - \phi(x_{k+1},\theta_k)  
			\geq \frac{1}{2} c_{\phi} \norm{d_k}^2. \\
  \end{aligned}
\end{equation*}
\end{lemma}
\begin{proof}
	From Lemma~\ref{lem:line-search-merit}, applying Assumption~\ref{assp:Bk} to~\eqref{eqn:line-search-cond}, we have 
\begin{equation} \label{eqn:merit-decrease-2} 
 \centering
  \begin{aligned}
	  \phi&(x_k,\theta_k) - \phi(x_{k+1},\theta_k)  
			\geq& \eta \alpha_k \frac{1}{2} b \norm{d_k}^2. \\
  \end{aligned}
\end{equation}
	From the proof of Lemma~\ref{lem:line-search-merit}, once $\alpha_k< b/(\rho+\theta_k c_h)$, the line search exits.
	Using ceiling function $\lceil\cdot \rceil$, which returns the least integer greater than the input, we have  
   \begin{equation} \label{eqn:merit-decrease-3}
   \centering
    \begin{aligned}
	    \alpha_k \geq \tau_{\alpha} ^{\lceil \log_{\tau_{\alpha}} \frac{b }{\rho +\theta_k  c_h} \rceil}.
    \end{aligned}
   \end{equation}
	where $\tau_{\alpha}$ is a parameter from the algorithm. From Lemma~\ref{lem:bounded-lp}, $\{\theta_k\}$ is bounded above.
	Therefore, there exists $\underline{\alpha}>0$ such that $\alpha_k\geq \underline{\alpha}$ for all $k=0,1,2,\dots$.
	Let $c_{\phi} = \eta \underline{\alpha} b$ and the proof is complete.
\end{proof}
%%%%%%%%%%%%%%%%%%%%%%%%%%%%%%%%%%%%%%%%%%%%%%%%%%%%
\begin{theorem}\label{thm:simp-KKT}
	If the MFCQ conditions in Lemma~\ref{lem:bounded-lp} stand, then every accumulation point $\bar{x}$ of $\{x_k\}$ generated by Algorithm~\ref{alg:sqp} 
	is a KKT point of the problem~\eqref{eqn:opt0}. 
	That is, there exists a subsequence $\{x_{k_s}\}$ of $\{x_k\}$, where $ x_{k_s} \to \bar{x}$, and $\bar{\lambda}\in\Rbb^m$, so that the KKT conditions~\eqref{eqn:opt0-KKT} are satisfied at $\bar{x}$.
\end{theorem}
\begin{proof}
	By Lemma~\ref{lem:bounded-lp}, there exists $k_0>0$ such that for $t\geq k_0$, the Lagrange multipliers are bounded above and $\theta_t=\theta_{k_0}=\bar{\theta}$. 
	Let $k\geq k_0$.
	Since $\{x_k\}$ and $\{g_{k}\}$ are bounded, there exists at least one accumulation point for $\{x_k\}$.
	Let $\bar{x}$ be an accumulation point of $\{x_k\}$ and $\{x_{k_s}\}$ be a subsequence of $\{x_k\}$ such that $x_{k_s}\to \bar{x}$.
	From Lemma~\ref{lem:merit-decrease}, we have that
	for $k$ large enough, $\{\phi(x_k,\theta_k)\}$ is a decreasing and bounded sequence 
	with a fixed parameter $\bar{\theta}$. 
	Thus, $\phi(x_k,\theta_k)$ converges. 
	
	From Lemma~\ref{lem:merit-decrease},
	we know that $\phi(x_k,\theta_k)-\phi(x_{k+1},\theta_k)$ is 
	bounded below in the order of $\norm{d_k}^2$. Therefore,  $\lim_{k\to\infty} \norm{d_k}\to 0$. In particular, $\lim_{s\to \infty}\norm{d_{k_s}}\to 0$.
	From  step 7 in Algorithm~\ref{alg:sqp} we have $\bar{\theta} \geq |\lambda_i^{k_s+1}| + \gamma$ for $i\in \mathcal{E}\cup\mathcal{I}$ and $k_s$ large enough. 
	Using the first equality in~\eqref{eqn:sqp-penal-KKT-2}, we have $p_i^{k_s+1}\geq \gamma>0$. Passing on to a subsequence if necessary, we assume $p^{k_s} \to \bar{p}>0$. Then, by the fourth equation in~\eqref{eqn:sqp-penal-KKT-2}, $v^{k_s}=0$. Similarly, we have $w^{k_s}= 0$ and $t^{k_s}= 0$ for $k$ large enough. From the third equation and fourth line in~\eqref{eqn:sqp-penal-KKT-1}, we have 
       \begin{align}
          c_i(\bar{x}) = 0, \ i \in \mathcal{E}, \ c_i(\bar{x}) \geq 0, \ i\in\mathcal{I}.
        \end{align}
          
	Passing on further to a subsequence if necessary, we let $g_{k_s}\to \bar{g}$, 
	$\lambda^{k_{s}+1} \to \bar{\lambda}$. 
        From the first equation in the optimality conditions~\eqref{eqn:sqp-penal-KKT-1}, we have 
        \begin{equation} \label{eqn:simp-bundle-KKT-limit-1}
          \centering
          \begin{aligned}
		  0 = \bar{g} - \sum_{i\in\mathcal{E}\cup\mathcal{I}} \bar{\lambda}_i \nabla c_i(\bar{x}). 
          \end{aligned}
        \end{equation}
	By the outer semicontinuity of Clarke subdifferential, with $g_{k_s}\in \bar{\partial} f(x_{k_s})$, we have $\bar{g} \in \bar{\partial} f(\bar{x})$.
        As a result, 
	  $0 \in \bar{\partial} f(\bar{x}) - \sum_{i\in\mathcal{E}\cup\mathcal{I}}\bar{\lambda}_i \nabla c_i (\bar{x}) $.
	Therefore, the necessary 
	optimality conditions~\eqref{eqn:opt0-KKT} at $\bar{x}$ are satisfied.
\end{proof}

%% file: Sections_journal/Algorithm_local.tex
\section{Upper-$\Ctwo$ and KL functions}\label{sec:KLfunctions}
To establish local convergence for nonsmooth optimization problems, additional assumptions such as subdifferential error bound~\cite{atenas2023bundlelocal}, KL functions~\cite{liu2019pdcae,yu2021KL,noll2014KL} are often made.  
For SQP, it is appealing to assume KL property on the merit function.
However, the nonsmooth (specifically, DC) nature of~\eqref{eqn:opt-sqp-merit} makes it difficult to obtain a (lower) general subgradient. 
Thus, we resort to a specially constructed potential function similar to~\cite{liu2019pdcae}.

As mentioned in Section~\ref{sec:prob}, for a bounded domain and Lipschitz upper-$\Ctwo$ $f(\cdot)$, there exists $\sigma>0$ such that $-f(x)+\frac{1}{2}\sigma\norm{x}^2$ is convex. 
From the bounded feasible region in Assumption~\ref{assp:Bk}, given a convex and compact set $D\subset \Rbb^n$ such that $\Omega \subseteq D$ and $\Omega_k\subseteq D$. We can define the following proper, lower semicontinuous, convex function 
\begin{equation}\label{def:F}
       F(x) := 
         \begin{cases}
           -f(x)+\frac{1}{2}\sigma\norm{x}^2, \ &x \in D,\\
           \infty, \ &\text{otherwise}.
         \end{cases}
\end{equation}
Thus, $\text{dom} \partial F = D$ and if $x\in D, g \in \partial^+ f(x)$, we know $-g +\sigma x \in \partial F(x)$.
A conjugate function $F^*$ of $F$ can be formulated by~\eqref{def:conjugate}. 
By Proposition~\ref{prop:conj} and the Lipschitz continuity of $F$ on $D$, we have $\text{dom} \partial F^*$ is closed and bounded~\cite[Theorem~9.13]{rockafellar1998}.

Define the linearized constraint function $\bar{c}_i(x,w) = c_i(w)+\nabla c_i(w)^T (x-w), i\in\mathcal{E}\cup\mathcal{I}$.
The approximated feasible region is
\begin{equation}\label{eqn:linear-feasible-region}
   \bar{C}(x,w) := \left\{ (x,w) \in \Rbb^n\times\Rbb^n \middle \vert
            \begin{array}{l}
               \bar{c}_i(x,w) = 0, \ i\in \mathcal{E}, \\
               \bar{c}_i(x,w) \geq 0, \ i\in \mathcal{I}. \\
            \end{array} 
           \right\}.
\end{equation}
To obtain  general subgradient of the indicator function $i_{\bar{C}}$, we shall make the additional assumption that functions $c_i(\cdot)$ are twice continuously differentiable when needed. 
The potential function is defined as 
\begin{equation}\label{eqn:opt1-KLfunction}
            L(x,y,w) = - y^T x + F^*(y) +\frac{1}{2} \sigma \norm{x}^2 +\frac{1}{2} l \norm{x-w}^2 + i_{\bar{C}}(x,w),
\end{equation}
where $l>0$ is a constant. Clearly, $L(\cdot,\cdot,\cdot)$ is lower semicontinuous.

To obtain an upper-$\Ctwo$ $f(\cdot)$ with a $L(\cdot,\cdot,\cdot)$ that is a KL function, we turn to the group of subanalytic functions.
For simplicity of presentation, we apply existing results in literature whenever possible.
From its definition, we know an upper-$\Ctwo$ $f(\cdot)$ on an open and subanalytic set $O \subset \Rbb^n$ can be expressed in the form of~\eqref{eqn:uppc2-def-1}.
If $p(\cdot,\cdot)$ and $T$ in~\eqref{eqn:uppc2-def-1} are additionally subanalytic objects, then $f$ is subanalytic on $O$~\cite[Remarks~3.11]{bierstone1988semianalytic}. With a bounded feasible region throughout optimization, there exists a convex and compact $D\subset O$ so that $f$ is subanalytic and upper-$\Ctwo$ with uniform $\rho$ in~\eqref{eqn:uppc2-def} on $D$ and $x_k\in D$ for all $k$.
Define $F(\cdot)$ of~\eqref{def:F} with such a $D$, then $F$ is subanalytic. 

Next, we note that if we let $\sigma > \rho$ and define $\mu = \sigma-\rho>0$, then $F$ is $\mu$ strongly convex~\cite[Exercise~12.59]{rockafellar1998}. The choice of $\sigma$ affects the optimization algorithm through conditions on $b$, and in turn $B_k$.
Given a $\mu$ strongly convex function, the following result holds.
\begin{proposition}\label{prop:convex-subanalytic}
  If $F$ is a lower semicontinuous, subanalytic and $\mu$ strongly convex function, then its conjugate function $F^*$ is continuously differentiable with Lipschitz continuous derivative, subanalytic and convex. 
\end{proposition}
Proposition~\ref{prop:convex-subanalytic} is established in~\cite{le2018KL,bolte2007lojasiewicz}. 
Therefore, provided with a subanalytic $f$ and an appropriate $\sigma$, $\rho$ and $b$, the $-y^T x+F^*(y)+\frac{1}{2}\sigma\norm{x}^2+\frac{1}{2}l\norm{x-w}^2$ part of $L(\cdot,\cdot,\cdot)$ is subanalytic and continuous on $\text{dom}\partial L$.  

If functions $c_i(\cdot),i\in\mathcal{E}\cup\mathcal{I}$ are analytic functions, 
then the feasible set $\Omega$ in~\eqref{eqn:feasible-set} together with its indicator function $i_{\Omega}(\cdot)$ are subanalytic sets and functions, respectively~\cite[Example~4.4]{bolte2007lojasiewicz}.  
Similarly, $i_{\bar{C}}(\cdot,\cdot)$ is subanalytic. Given that subanalytic sets are closed under locally finite union and intersection, 
$L$ is subanalytic on $\text{dom}\partial L$. Finally, a continuous subanalytic function with closed domain is a KL function with exponent $\alpha\in [0,1)$~\cite[Theorem~3.1]{bolte2007lojasiewicz}.
We formalize the result below. 
\begin{theorem}\label{thm:upperC2andKL}
    Given Assumption~\ref{assp:upperC2},~\ref{assp:boundedHc} and~\ref{assp:Bk}, let $D$ be a subanalytic, convex, compact set 
    so that $\Omega \subset D$ and $\Omega_k\subset D, k=1,2,\dots$ and define $F$ in~\eqref{def:F} with such a $D$.
    If the optimization problem~\eqref{eqn:opt0} satisfies on $D$ the following properties\\ 
   \indent (\expandafter{\romannumeral 1}) the upper-$\Ctwo$ objective can be expressed as~\eqref{eqn:uppc2-def-1} with subanalytic $p(\cdot,\cdot)$ and $T$,\\
   \indent (\expandafter{\romannumeral 2}) the constant $\sigma$ in~\eqref{eqn:opt1-KLfunction} and $\rho$ in~\eqref{eqn:uppc2-convex} satisfies $\sigma>\rho$,\\ 
   \indent (\expandafter{\romannumeral 3}) the constraint functions $c_i, i\in\mathcal{E}\cup\mathcal{I}$ are analytic,\\ 
    then the potential function $L$ defined as~\eqref{eqn:opt1-KLfunction} is a KL function with exponent $\alpha\in [0,1)$ on $ \text{dom} \partial L$.  
\end{theorem}
Theorem~\ref{thm:upperC2andKL} implies that assuming a KL potential function for an upper-$\Ctwo$ objective is reasonable in many cases, including for power-grid optimization problems where the functions are bounded and analytic.

\section{\normalsize Local convergence analysis}\label{sec:local}

In this section, we will study the local convergence properties of Algorithm~\ref{alg:sqp}.
Our goal is to maintain the global convergence established in Section~\ref{sec:alg-convg} without modifying the algorithm itself, and focus on the local convergence property when the iteration $k$ is large enough and $\{x_k\}$ is close to a KKT point, per convention of SQP local convergence analysis.
Therefore, Assumptions~\ref{assp:upperC2},~\ref{assp:boundedHc},~\ref{assp:Bk} and the MFCQ conditions in Lemma~\ref{lem:bounded-lp} are assumed valid throughout the analysis. Consequently, the Lemmas and Theorems in Section~\ref{sec:alg-convg} remain valid. 

It is important to clarify what a large enough $k$ means. For the algorithmic parameters, we have $\theta_k=\bar{\theta}$ by Lemma~\ref{lem:bounded-lp} for $k$ large enough.
Theorem~\ref{thm:simp-KKT} shows that the linearized constraint in the quadratic subproblem is consistent for $k$ large enough, \textit{i.e.}, 
$v^{k} = 0$, $w^{k} = 0$ and $t^k=0$ for $k$ large enough. 
We summarize the properties when $k$ is large enough in the following proposition.
\begin{proposition}\label{prop:large-k}
    There exists integer $k_0>0$ such that Algorithm~\ref{alg:sqp} applied to~\eqref{eqn:opt0} with~\eqref{eqn:opt-sqp-penal} 
    satisfies:  $\theta_k=\bar{\theta}$, $v^k=w^k=t^k=0$ for all $k>k_0$. 
\end{proposition}

For most SQP methods with line search, a consistent local convergence rate requires a full step, \textit{i.e.}, $\alpha_k=1$ in $x_{k+1}=x_k+\alpha_k d_k$~\cite{boggs1995sqp}. A natural transition to a full step is not guaranteed with the classic SQP in Algorithm~\ref{alg:sqp}.
Therefore, we divide the local convergence analysis into two cases. The first case is when the constraints are affine, where for $k$ large enough, $\alpha_k=1$ is accepted. The second case is for the general smooth constraints, where additional assumptions are needed for a full step.
% equality only 
%Furthermore, it is common for SQP local convergence analysis to focus on equality constrained problems for simplicity. This is justified because the active set of inequality constraint for the subproblems when $k$ large is enough is the same as that of a KKT point in the neighborhood~\cite{boggs1995sqp}, when additional second-order conditions and strict complementarity conditions are satisfied~\cite{Nocedal_book}.  
%The specific conditions for a stabilized active inequality constraint are beyond the scope of this paper. 
%Instead, we will also focus on the equality constrained optimization problem.  
%For brevity of presentation, we use the notations and results from previous Sections with $I=\emptyset$ throughout the local convergence analysis. 

\subsection{Local convergence with affine constraints}\label{sec:affine-local}
In this subsection, we assumed that the constraints in~\eqref{eqn:opt0} are affine. Therefore, 
the linearized constraints in optimization subproblems~\eqref{eqn:opt-sqp-penal} are exact. This assumption eliminates the constraint terms in~\eqref{eqn:opt-sqp-merit} at subproblem solutions $x_k+d_k$. 
While affine constraints are assumed here, a variety of assumptions and algorithms can be adopted to achieve the 
goal of generating a feasible sequence $\{x_k\}$ for~\eqref{eqn:opt0}. 
One such assumption is that the feasible region $\Omega$ is convex combined with a projected quadratic programming algorithm.
Another option is to have quadratically approximated inequality constraints that enforces feasibility of each $x_k$~\cite{yu2021KL}.
The assumptions of this section are formalized below. 
\begin{assumption}\label{assp:affine}
   The constraint functions $c_i, i\in \mathcal{E}\cup\mathcal{I}$ are affine. 
   The potential function~\eqref{eqn:opt1-KLfunction} is a KL function.
\end{assumption}
Assumption~\ref{assp:affine} is considered valid in the following analysis.
First, we prove the following Lemma similar to that of Lemma~\ref{lem:line-search-eqcons} under affine constraints.
\begin{lemma}\label{lem:line-search-affine-eqcons}
	For $k$ large enough, the equality constraints satisfy the following inequality 
       \begin{equation} \label{eqn:line-search-affine-eq}
         \centering
         \begin{aligned}
		 \bar{\theta} \sum_{i\in\mathcal{E}}|c_i(x_k)| - \bar{\theta} \sum_{i\in\mathcal{E}}|c_i(x_{k+1})| - \alpha_k \sum_{i\in\mathcal{E}} \lambda^{k+1}_i\nabla c_i(x_k)^T d_k
                     \geq 0.
      \end{aligned}
     \end{equation}
\end{lemma}
\begin{proof} 
     By Assumption~\ref{assp:affine}, for $x_{k+1}=x_k+\alpha_k d_k$,
        \begin{equation} \label{eqn:sqp-ls-aff-pf-1}
         \centering
         \begin{aligned}
		 c_i(x_{k+1})  = c_i(x_k) + \alpha_k \nabla c_i(x_k)^T d_{k}.\\
      \end{aligned}
     \end{equation}
	By triangle inequality, Proposition~\ref{prop:large-k} and third line in~\eqref{eqn:sqp-penal-KKT-1}, we have
        \begin{equation} \label{eqn:sqp-ls-aff-pf-2}
   \centering
    \begin{aligned}
	    |c_i(x_{k+1})|  
	    & = | c_i(x_k) + \alpha_k \nabla c_i(x_k)^T d_{k}|
	     = \left|(1-\alpha_k) c_i(x_k) + \alpha_k \left[c_i(x_k)+ \nabla c_i(x_k)^T d_{k}\right]\right|\\ 
	      & \leq (1-\alpha_k)| c_i(x_k)| + \alpha_k |c_i(x_k)+ \nabla c_i(x_k)^T d_{k}| 
               =  (1-\alpha_k)| c_i(x_k)|.
    \end{aligned}
   \end{equation}
   Summing~\eqref{eqn:sqp-ls-aff-pf-2} over $i\in\mathcal{E}$, we have 
      \begin{equation} \label{eqn:sqp-ls-aff-pf-3}
     \centering
      \begin{aligned}
	      \sum_{i\in\mathcal{E}} |c_i(x_k)|-\sum_{i\in\mathcal{E}} |c_i(x_{k+1})|\geq  
	    \alpha_k  \sum_{i \in \mathcal{E}}  |c_i(x_k)|. \\
      \end{aligned}
    \end{equation}
    From Proposition~\ref{prop:large-k}, Lemma~\ref{lem:lambda-sqp-prop} and~\eqref{eqn:sqp-penal-KKT-1}, we can write
      \begin{equation} \label{eqn:sqp-ls-aff-pf-4}
     \centering
      \begin{aligned}
	      \sum_{i\in \mathcal{E}} \lambda_i^{k+1}\nabla c_i(x_k)^T d_k =& 
	       -\sum_{i\in \mathcal{E}} \lambda^{k+1}_i  c_i(x_k) \leq \bar{\theta}\sum_{i\in \mathcal{E}}  |c_i(x_k)|.\\
      \end{aligned}
    \end{equation}
    Therefore, combining~\eqref{eqn:sqp-ls-aff-pf-3} and~\eqref{eqn:sqp-ls-aff-pf-4}, we have 
      \begin{equation} \label{eqn:sqp-ls-aff-pf-5}
   \centering
    \begin{aligned}
	    \bar{\theta} \sum_{i\in\mathcal{E}} |c_i(x_k)|-\bar{\theta}\sum_{i\in\mathcal{E}} |c_i(x_{k+1})| &- \alpha_k \sum_{i\in\mathcal{E}}\lambda^{k+1}_i\nabla c_i(x_k)^T d_k \geq\\ 
             &\bar{\theta} \alpha_k  \sum_{i \in \mathcal{E}}  |c_i(x_k)|- 
                  \bar{\theta} \alpha_k\sum_{i\in \mathcal{E}}  |c_i(x_k)|  = 0.
     \end{aligned}
    \end{equation}
\end{proof}
Similarly, the following Lemma stands for the inequality constraints.
\begin{lemma}\label{lem:line-search-affine-ineqcons}
	For $k$ large enough, the inequality constraints satisfy
       \begin{equation} \label{eqn:line-search-affine-ineq}
         \centering
         \begin{aligned}
		 \bar{\theta} \sum_{i\in\mathcal{I}} [c_i(x_k)]^{-} - \bar{\theta} \sum_{i\in\mathcal{I}} [c_i(x_{k+1})]^{-} 
		  - \alpha_k \sum_{i\in\mathcal{I}} \lambda^{k+1}_i\nabla c_i(x_k)^T d_k
                     \geq 0.
      \end{aligned}
     \end{equation}
\end{lemma}
\begin{proof}
      By Assumption~\ref{assp:affine}, for $x_{k+1}=x_k+\alpha_k d_k$ and $i\in\mathcal{I}$,
        \begin{equation} \label{eqn:sqp-ineq-aff-pf-1}
         \centering
         \begin{aligned}
		 \relax [c_i(x_{k+1})]^-  =  [c_i(x_k)+ \alpha_k \nabla c_i(x_k)^T d_{k}]^- .\\
      \end{aligned}
     \end{equation}
     Multiplying by $-1$ and adding $[c_i(x_k)]^-$ to~\eqref{eqn:sqp-ineq-aff-pf-1} gives us 
        \begin{equation} \label{eqn:sqp-ineq-aff-pf-2}
         \centering
         \begin{aligned}
	\relax [c_i(x_{k})]^- - [c_i(x_{k+1})]^- = [c_i(x_{k})]^--[c_i(x_k)+ \alpha_k \nabla c_i(x_k)^T d_{k}]^-.\\
      \end{aligned}
     \end{equation}
     We notice that for $k$ large enough, $V_k^i = \emptyset$ and $A_k^i=\mathcal{I}$ by definition~\eqref{eqn:opt-sqp-fea-A}.
     We consider two cases. In the first case, $c_i(x_k)+ \nabla c_i(x_k)^T d_{k} =0 $. Applying it to~\eqref{eqn:sqp-ineq-aff-pf-2}, we obtain  
      \begin{equation} \label{eqn:sqp-ineq-aff-pf-3}
         \centering
         \begin{aligned}
	\relax [c_i(x_{k})]^- - [c_i(x_{k+1})]^- 
		   \geq   \alpha_k [c_i(x_{k})]^- .\\
      \end{aligned}
     \end{equation}
     Therefore, by Lemma~\ref{lem:lambda-sqp-prop}, 
      \begin{equation} \label{eqn:sqp-ineq-aff-pf-4}
         \centering
         \begin{aligned}
		 \bar{\theta} [c_i(x_{k})]^- - \bar{\theta} [c_i(x_{k+1})]^- -  \alpha_k \lambda_i^{k+1} \nabla c_i(x_k)^T d_k
		 \geq \alpha_k \bar{\theta}[c_i(x_{k})]^- + \alpha_k \lambda_i^{k+1} c_i(x_k) \geq 0.\\
      \end{aligned}
     \end{equation}
      In the second case, 
	$c_i(x_k)+ \nabla c_i(x_k)^T d_{k} >0 $. It is simple to verify that whatever the signs of $c_i(x_k)$ and $\nabla c_i(x_k)^T d_{k}$ are, 
       \begin{equation} \label{eqn:sqp-ineq-aff-pf-4.5}
         \centering
         \begin{aligned}
      	   \relax [c_i(x_{k})]^- -[c_i(x_k)+ \alpha_k \nabla c_i(x_k)^T d_{k}]^- \geq 0.
           \end{aligned}
     \end{equation}
      By~\eqref{eqn:sqp-ineq-aff-pf-2},
      \begin{equation} \label{eqn:sqp-ineq-aff-pf-5}
         \centering
         \begin{aligned}
      	  \relax [c_i(x_{k})]^- -  [c_i(x_{k+1})]^-  \geq  0.
           \end{aligned}
     \end{equation}
      From the complementarity condition in~\eqref{eqn:sqp-penal-KKT-1}, $\lambda_i^{k+1}=0$. Thus, 
      \begin{equation} \label{eqn:sqp-ineq-aff-pf-6}
         \centering
         \begin{aligned}
           \bar{\theta} [c_i(x_{k})]^- - \bar{\theta} [c_i(x_{k+1})]^- -  \alpha_k \lambda_i^{k+1} \nabla c_i(x_k)^T d_k
		 = \bar{\theta} [c_i(x_{k})]^- - \bar{\theta} [c_i(x_{k+1})]^- \geq 0.
        \end{aligned}
     \end{equation}
     Combine the two cases and sum over $i\in\mathcal{I}$, and the lemma is proven. 
\end{proof}
Next, we show that the step size $\alpha_k=1$ at $x_k$ meets the line search criterion. 
\begin{lemma}\label{lem:affine-step-size}
   Let $b>\rho$. The line search condition~\eqref{eqn:line-search-cond} is satisfied with $\alpha_k=1$ for $k$ large enough.
\end{lemma}
\begin{proof}
    We omit steps already explained in the proof of Lemma~\ref{lem:line-search-merit}. By Proposition~\ref{prop:large-k}, $\phi(x_k,\theta_k) = \phi(x_k,\bar{\theta})$. From~\eqref{eqn:sqp-penal-merit-pf-3}, 
      Lemma~\ref{lem:line-search-affine-eqcons} and Lemma~\ref{lem:line-search-affine-ineqcons},
    \begin{equation} \label{eqn:sqp-affine-step-size-pf-1} 
 \centering
  \begin{aligned}
	  \phi&(x_k,\bar{\theta}) - \phi(x_{k+1},\bar{\theta})  = f(x_k) -f(x_{k+1})+ \bar{\theta} v(x_k)-\bar{\theta} v(x_{k+1}) \\
			 \geq & \alpha_k d_k^TB_kd_k - \frac{\rho}{2} \alpha_k^2\norm{d_k}^2- \alpha_k \sum_{i\in \mathcal{E}\cup\mathcal{I}} \lambda_i^{k+1} \nabla c_i(x_k)^T d_k + \bar{\theta} v(x_k)-\bar{\theta} v(x_{k+1}) \\
			\geq& \frac{1}{2}\alpha_k d_k^TB_kd_k + \frac{\alpha_k}{2} (b- \rho \alpha_k)  \norm{d_k}^2.
  \end{aligned}
\end{equation}
    If $\alpha_k=1$ and $b>\rho$, then the line search condition~\eqref{eqn:line-search-cond} is satisfied for $\eta\leq 1$.
\end{proof}
\begin{remark}
 The constant $b$ is entirely controlled by the choice of $B_k$ while $\rho$ is a constant of the function $f$ on $D$. Therefore, the condition $b>\rho$ is reasonable and implementable. One practical algorithm to ensure this is given in the author's previous paper~\cite{wang2022}. Similarly, we ask for inequalities among $b$, $\rho$, $\sigma$, $l$ throughout the local convergence analysis, all of which are implementable. The one condition that is not easily achieved is given as an assumption in Assumption~\ref{assp:Bk-c}. 
\end{remark}
Lemma~\ref{lem:affine-step-size} establishes the sufficient decrease for merit function. 
Next, we establish the decreasing property for~\eqref{eqn:opt1-KLfunction}. 
\begin{lemma}\label{lem:affine-lag-decrease}
   Let $2b\geq \sigma+l$ and $b\geq \sigma\geq \rho$. For $k$ large enough, there exists constants $c_d>0$ such that 
   \begin{equation} \label{eqn:affine-lag-0}
     \begin{aligned}
         \centering
	L(x_{k}, -g_{k-1}+\sigma x_{k-1} , x_{k-1})-L(x_{k+1}, -g_{k}+\sigma x_{k},x_{k}) \geq  \frac{1}{2} c_d \norm{d_{k-1}}^2.
      \end{aligned}
     \end{equation}
\end{lemma}
\begin{proof}
    By the second equation in~\eqref{eqn:sqp-penal-KKT-1}, Proposition~\ref{prop:large-k}, Lemma~\ref{lem:affine-step-size}, we have 
  \begin{equation} \label{eqn:affine-lag-1}
   \centering
     \begin{aligned}
           c_i(x_k) = c_i(x_{k-1}+ d_{k-1}) = c_i(x_{k-1})+ \nabla c_i(x_{k-1})^T d_{k-1} =0, \ i\in\mathcal{E},\\
           c_i(x_k) = c_i(x_{k-1}+ d_{k-1}) = c_i(x_{k-1})+ \nabla c_i(x_{k-1})^T d_{k-1} \geq 0, \ i\in\mathcal{I}.\\
     \end{aligned}
   \end{equation}
   Similarly, $c_i(x_{k+1}) = 0, i\in\mathcal{E}$ and $c_i(x_{k+1}) \geq 0, i\in\mathcal{I}$.  
   To simplify notations, denote $ z_k = -g_k+\sigma x_k$.
 From the definition~\eqref{eqn:opt1-KLfunction},~\eqref{eqn:affine-lag-1} and Lemma~\ref{lem:affine-step-size}, we can write 
    \begin{equation} \label{eqn:affine-lag-2}
         \centering
         \begin{aligned}
L(x_{k+1}, z_k , x_k)-L(x_{k}, z_{k-1},x_{k-1})&
             = -z_k^T x_{k+1}+F^*(z_k) +\frac{1}{2}\sigma\norm{x_{k+1}}^2 +\frac{1}{2}l \norm{d_k}^2\\
             & +z_{k-1}^T x_{k}-F^*(z_{k-1}) -\frac{1}{2}\sigma\norm{x_{k}}^2 -\frac{1}{2}l\norm{d_{k-1}}^2.\\
      \end{aligned}
     \end{equation}
    Since $F(\cdot)$ is convex and $z_k \in \partial F(x_k)$, by~\eqref{def:conjugate-subgradient}, $F(x_k)+F^*(z_k)=z_k^T x_k$. Similarly,  $F(x_{k-1})+F^*(z_{k-1})= z_{k-1}^T x_{k-1}$. Applying these equalities to~\eqref{eqn:affine-lag-2}, we obtain
    \begin{equation} \label{eqn:affine-lag-3}
         \centering
         \begin{aligned}
L(x_{k+1}, &z_k , x_k)-L(x_{k}, z_{k-1},x_{k-1})
             = -z_k^T x_{k+1}+z_k^T x_k-F(x_k) +\frac{1}{2}\sigma\norm{x_{k+1}}^2 \\
             & \qquad +\frac{1}{2}l\norm{d_k}^2+z_{k-1}^T x_{k}+F(x_{k-1})-z_{k-1}^T x_{k-1}   -\frac{1}{2}\sigma\norm{x_{k}}^2 -\frac{1}{2}l\norm{d_{k-1}}^2\\
             =&  g_k^T d_k- \sigma x_k x_{k+1} +f(x_k) +\frac{1}{2}\sigma\norm{x_{k+1}}^2 +\frac{1}{2}l\norm{d_k}^2 - g_{k-1}^T d_{k-1}\\
             & +\sigma x_k^T x_{k-1}-\frac{1}{2}\sigma\norm{x_{k-1}}^2 -f(x_{k-1}) -\frac{1}{2}l\norm{d_{k-1}}^2\\
             =&  g_k^T d_k+ \frac{1}{2}(l+\sigma) \norm{d_k}^2  +f(x_k) 
              - g_{k-1}^T d_{k-1} -f(x_{k-1}) -\frac{1}{2}(l+\sigma)\norm{d_{k-1}}^2.\\
      \end{aligned}
     \end{equation}
 From here, we follow steps from the proof of Lemma~\ref{lem:line-search-merit}. 
    Applying Assumption~\ref{assp:upperC2}, the optimality conditions~\eqref{eqn:sqp-penal-KKT-1} and Proposition~\ref{prop:large-k}, we have 
    \begin{equation} \label{eqn:affine-lag-4}
         \centering
         \begin{aligned}
      L(&x_{k+1}, z_k , x_k)-L(x_{k}, z_{k-1},x_{k-1})\\
             &\leq \frac{\rho}{2} \norm{d_{k-1}}^2 +g_k^T d_k +\frac{1}{2}(l+\sigma)\norm{d_k}^2-\frac{1}{2}(l+\sigma)\norm{d_{k-1}}^2\\
      &= -d_k^T B_d d_k +\sum_{i\in\mathcal{E}\cup\mathcal{I}}\lambda^{k+1}\nabla c_i(x_k)^T d_k+\frac{1}{2}(l+\sigma)\norm{d_k}^2-\frac{1}{2}(l+\sigma-\rho) \norm{d_{k-1}}^2\\
      &\leq -b \norm{d_k}^2 -\sum_{i\in\mathcal{E}\cup\mathcal{I}}\lambda^{k+1} c_i(x_k)+\frac{1}{2}(l+\sigma)\norm{d_k}^2-\frac{1}{2}(l+\sigma-\rho)\norm{d_{k-1}}^2\\
         &  \leq   -\frac{1}{2}(2b-l-\sigma) \norm{d_k}^2-\frac{1}{2}(l+\sigma-\rho) \norm{d_{k-1}}^2.
      \end{aligned}
     \end{equation}
   The first inequality of~\eqref{eqn:affine-lag-4} comes from Assumption~\ref{assp:upperC2}. The equality comes from the first line of~\eqref{eqn:sqp-penal-KKT-1} multiplied by $d_k$. The second inequality is due to complementarity conditions in~\eqref{eqn:sqp-penal-KKT-1} and Proposition~\ref{prop:large-k}. The third inequality stands by~\eqref{eqn:sqp-penal-KKT-1} and~\eqref{eqn:affine-lag-1}. The Lemma is proven with $2b\geq l+\sigma$ and $b\geq \sigma\geq \rho$. 
\end{proof}

Some important properties of  $L(\cdot,\cdot,\cdot)$ along the sequences $\{x_k\}$ and $\{g_k\}$ generated by the algorithm is presented in the next lemma. 
\begin{lemma}\label{lem:aux-bound-1}
   Let $U$ be the set of accumulation points of $\{x_{k+1},-g_k+\sigma x_k,x_k\}$. Then, $U$ is nonempty and compact. 
   Given conditions in Lemma~\ref{lem:affine-lag-decrease}, the sequence $\{L(x_{k+1},-g_k+\sigma x_k,x_k)\}$ converges and maintains the same value on $U$.
\end{lemma}
\begin{proof}
     Since $\{x_k\}$ and $\{g_k\}$ are bounded, there exists at least one accumulation point of the sequence $\{x_{k+1},-g_k+\sigma x_k,x_k\}$. The first part of the lemma is proven. 
    
     From~\eqref{def:conjugate-subgradient}, we know that 
    \begin{equation}\label{eqn:aux-bound-1}
     \centering
     \begin{aligned}
       L(x,y,w) \geq -F(x)+\frac{1}{2}\sigma \norm{x}^2+\frac{1}{2}l\norm{x-w}^2+ i_{\bar{C}}(x,w) = f(x)+\frac{1}{2}l\norm{x-w}^2+ i_{\bar{C}}(x,w).
      \end{aligned}
    \end{equation}
    Denote $z_k=-g_k+\sigma x_k$. Since $i_{\bar{C}}(x_{k+1},x_k)=0$, the sequence  $\{L(x_{k+1},z_k,x_k)\}$ is bounded below. By Lemma~\ref{lem:affine-lag-decrease}, $\{L(x_{k+1},z_k,x_k)\}$ is non-increasing, and hence convergent. 
    Define $\bar{L} = \lim_{k\to\infty} L(x_{k+1},z_k,x_k)$.
    For any $(\bar{x},\bar{z},\hat{x}) \in U$ where $\bar{z}=-\hat{g}+\sigma \hat{x}$, there exists a subsequence $\lim_{k_s \to\infty} (x_{k_s+1},z_{k_s},x_{k_s})\to (\bar{x},\bar{z},\hat{x})$. By applying $\norm{d_k}\to 0$ repeatedly, we can write 
    \begin{equation}\label{eqn:aux-bound-2}
     \centering
     \begin{aligned}
      \bar{L}=&\lim_{k_s \to\infty} L(x_{k_s+1}, z_{k_s},x_{k_s}) = \lim_{k_s \to\infty}-z_{k_s}^T x_{k_s+1}+F^*(z_{k_s})  +\frac{1}{2}\sigma\norm{x_{k_s+1}}^2\\
          &\qquad\qquad\qquad\qquad\qquad\qquad\qquad\quad +\frac{1}{2}l\norm{d_{k_s}}^2+i_{\bar{C}}(x_{k_s+1},x_{k_s})\\
           =& \lim_{k_s \to\infty} -z_{k_s}^T x_{k_s+1} +z_{k_s}^T x_{k_s}- F(x_{k_s}) +\frac{1}{2}\sigma\norm{x_{k_s+1}}^2\\
           =& \lim_{k_s \to\infty} -z_{k_s}^T d_{k_s} + f(x_{k_s})-\frac{1}{2}\sigma\norm{x_{k_s}}^2 +\frac{1}{2}\sigma\norm{x_{k_s+1}}^2\\
           =& \lim_{k_s \to\infty}  f(x_{k_s}) -\frac{1}{2}\sigma (x_{k_s}+x_{k_s+1})^T(x_{k_s}-x_{k_s+1}) = \lim_{k_s \to\infty}  f(x_{k_s}) \\
           =&f(\bar{x}) \leq L(\bar{x},\bar{z},\hat{x}).
      \end{aligned}
    \end{equation}
     Since $L(\cdot,\cdot,\cdot)$ is lower semicontinuous, we have  
    \begin{equation}\label{eqn:aux-bound-3}
     \centering
     \begin{aligned}
      \bar{L} = \liminf_{k_s\to\infty} L(x_{k_s+1},-g_{k_s}+\sigma x_{k_s},x_{k_s}) \geq L(\bar{x},\bar{z},\hat{x}). 
      \end{aligned}
    \end{equation} 
    Therefore, $L(\bar{x},\bar{z},\hat{x}) = \bar{L}$ 
\end{proof}
The subgradient property of~\eqref{eqn:opt1-KLfunction} is examined next.
\begin{lemma}\label{lem:lag-subgradient-1}
   There exists constant $c_{L}>0$  such that the subgradient of the potential function~\eqref{eqn:opt1-KLfunction} satisfies 
    \begin{equation}\label{eqn:aux-sub-1}
     \centering
     \begin{aligned}
        \mathrm{dist}(0, \partial L(x_{k+1},-g_k+\sigma x_k,x_{k})) \leq c_L \norm{d_k},  
      \end{aligned}
    \end{equation} 
   for all $k$ large enough. 
\end{lemma}
\begin{proof}
    For $(x,y,w)\in \text{dom} \partial L$, applying~\cite[Exercise~8.8]{rockafellar1998},~\cite[Proposition~10.5]{rockafellar1998} 
    and~\cite[Corollary~10.9]{rockafellar1998}, we can write
    \begin{equation}\label{eqn:aux-affine-sub-pf-1}
     \centering
     \begin{aligned}
         \begin{pmatrix} 
          -y  + \sigma x+l(x-w)  \\
          -x + \hat{\partial} F^*(y)\\
          -l(x-w)
         \end{pmatrix} + 
           \hat{\partial} i_{\bar{C}}(x,w) 
          \subseteq   \hat{\partial} L(x,y,w) \subseteq \partial L(x,y,w). 
      \end{aligned}
    \end{equation}
    Since $F^*$ is convex, by~\cite[Example~6.10, Theorem~6.14, Exercise~8.14]{rockafellar1998}, 
    we have   
    \begin{equation}\label{eqn:aux-affine-sub-pf-2}
     \centering
     \begin{aligned}
         \hat{\partial} F^*(y) = \partial F^*(y), \
         \sum_{i\in\mathcal{E}\cup \mathcal{I}} \lambda_i \nabla \bar{c}_i(x,w) \in \hat{\partial} i_{\bar{C}}(x,w),  
       \end{aligned}
    \end{equation}
    where $\lambda_i\in N_{\Rbb}(\bar{c}_i(x,w))$ for $i\in\mathcal{E}$ and $\lambda_i\in N_{\Rbb_+}(\bar{c}_i(x,w))$ for $i\in\mathcal{I}$.  
    Substituting~\eqref{eqn:aux-affine-sub-pf-2} in~\eqref{eqn:aux-affine-sub-pf-1} leads to 
    \begin{equation}\label{eqn:aux-affine-sub-pf-3}
     \centering
     \begin{aligned}
         \begin{pmatrix} 
          -y  + \sigma x+l(x-w)  \\
          -x + \partial F^*(y)\\
          -l(x-w)
         \end{pmatrix} + 
          \sum_{i\in\mathcal{E}\cup\mathcal{I}} \lambda_i 
         \begin{pmatrix}
           \nabla c_i(w)\\
           0\\
           -\nabla^2 c_i(w) (x-w)
         \end{pmatrix} 
          \subseteq   \partial L(x,y,w). 
      \end{aligned}
    \end{equation}
   We note that affine functions $c_i(\cdot)$ are by definition twice continuously differentiable
   and  $(x_{k+1},-g_k+\sigma x_k,x_{k}) \in \text{dom}\partial L$.
   At $x=x_{k+1},y = -g_k+\sigma x_k, w=x_k$, the first equation in~\eqref{eqn:sqp-penal-KKT-1} gives us  
    \begin{equation}\label{eqn:aux-affine-sub-pf-4}
     \centering
     \begin{aligned}
       g_k  - \sum_{i\in \mathcal{E}\cup\mathcal{I}} \lambda^{k+1}_i \nabla c_i(x_k) = -B_k d_k.
     \end{aligned}
    \end{equation}
    Let $\lambda_i=-\lambda^{k+1}_i$, $i\in\mathcal{E}\cup\mathcal{I}$. Notice that 
$-\lambda_i^{k+1}\in N_{\Rbb}(\bar{c}_i(x,w))$ for $i\in\mathcal{E}$ and $-\lambda_i^{k+1}\in N_{\Rbb_+}(\bar{c}_i(x,w))$ for $i\in\mathcal{I}$.
    Then, by Lemma~\ref{lem:affine-step-size}, the first line in~\eqref{eqn:aux-affine-sub-pf-3} satisfies  
    \begin{equation}\label{eqn:aux-affine-sub-pf-5}
     \centering
     \begin{aligned}
       g_k -\sigma x_k +\sigma x_{k+1} + l d_k  - \sum_{i\in \mathcal{E}\cup\mathcal{I}} \lambda^{k+1}_i \nabla c_i(x_k) = -B_k d_k+\sigma d_k.
     \end{aligned}
    \end{equation}
    Next, from~\eqref{def:conjugate-subgradient},   
    \begin{equation}\label{eqn:aux-affine-sub-pf-6}
     \centering
     \begin{aligned}
      -x_{k+1}+x_k \in  -x_{k+1} + \partial F^*(-g_k+\sigma x_k)   .
     \end{aligned}
    \end{equation}
    The third line in~\eqref{eqn:aux-affine-sub-pf-3} becomes
    \begin{equation}\label{eqn:aux-affine-sub-pf-7}
     \centering
     \begin{aligned}
         -l d_k - \sum_{i\in \mathcal{E}\cup\mathcal{I}} \lambda^{k+1}_i \nabla^2 c_i(x_k) d_k.
     \end{aligned}
    \end{equation}
     Combining~\eqref{eqn:aux-affine-sub-pf-5},~\eqref{eqn:aux-affine-sub-pf-6} and~\eqref{eqn:aux-affine-sub-pf-7} into~\eqref{eqn:aux-affine-sub-pf-3},  
    \begin{equation}\label{eqn:aux-affine-sub-pf-8}
     \centering
     \begin{aligned}
        \begin{pmatrix} 
          - B_k d_k +\sigma d_k\\
          -d_k \\
          -l d_k - \sum_{i\in \mathcal{E}\cup\mathcal{I}} \lambda^{k+1}_i \nabla^2 c_i(x_k) d_k
         \end{pmatrix}  
               \subseteq \partial L(x_{k+1},-g_k+\sigma x_k,x_k)
   \end{aligned}
    \end{equation}
    Since $B_k$, $\lambda^{k+1}$ and $\nabla^2 c(x_k)$ are bounded, there exists $c_L>0$ such that~\eqref{eqn:aux-sub-1} is satisfied.
\end{proof}
We are now ready to prove the convergence of $\{x_k\}$.
\begin{theorem}\label{thm:local-affine-KL}
    Let $b\geq \sigma\geq \rho$ and $2b\geq \sigma+ l$. The sequence $\{x_k\}$ produced by Algorithm~\ref{alg:sqp}  
    converges to a KKT point of~\eqref{eqn:opt0}. 
\end{theorem}
\begin{proof}
    Since the algorithm is not modified, we can directly use the results from Section~\ref{sec:alg-convg}.  
    From the proof of Theorem~\ref{thm:simp-KKT}, we know that $\{d_k\}\to 0$ and $\phi(x_k,\bar{\theta})$ converges as it is monotonically non-decreasing and bounded. Similarly, $L(x_{k+1},-g_k+\sigma x_k, x_k)$ converges by Lemma~\ref{lem:affine-lag-decrease}.
    Let $\phi(x_k,\bar{\theta})\to\bar{\phi}$ and $L(x_{k+1},-g_k+\sigma x_k, x_k)\to \bar{L}$. 
    We consider only the case $L(x_{k+1},-g_k+\sigma x_k, x_k) > \bar{L}$ for all $k$, since otherwise the proof is trivial.

      Let $U$ denote the set of accumulation points of $\{x_{k+1},-g_k+\sigma x_k,x_k\}$  and $\{\bar{x},-\hat{g}+\sigma \hat{x},\hat{x}\}\in  U$ be an accumulation point. 
 Since $\phi(\cdot,\bar{\theta})$ is Lipschitz continuous in $x$, we have $\phi(\bar{x},\bar{\theta})=\bar{\phi}$ on $U$. 
      In addition, Lemma~\ref{lem:aux-bound-1} guarantees $L(\bar{x},-\hat{g}+\sigma \hat{x},\hat{x}) = \bar{L}$, for all $\{\bar{x},-\hat{g}+\sigma \hat{x},\hat{x}\}\in U$. 

      By the KL property of $L(\cdot,\cdot,\cdot)$, there exists $\epsilon, a>0$ and a continuously differentiable concave function $\varphi(\cdot)$ such that 
\begin{equation} \label{eqn:local-affine-KL-1}
 \centering
  \begin{aligned}
	  \varphi'(L(x,y,w) - \bar{L})  \cdot \text{dist}(0, \partial L(x,y,w)) \geq 1, 
  \end{aligned}
\end{equation}
	for $\{x,y,w\} \in V$, where 
\begin{equation} \label{eqn:local-affine-KL-2}
 \centering
  \begin{aligned}
	  V = \{(x,y,w):\text{dist}((x,y,w),U)<\epsilon \} \cap \{(x,y,w): \bar{L} < L(x,y,w)<\bar{L}+a\}. 
  \end{aligned}
\end{equation}
    Since $U$ is the set of the accumulation points, we have 
\begin{equation} \label{eqn:local-affine-KL-3}
 \centering
  \begin{aligned}
       \lim_{k\to\infty}  \text{dist}((x_{k+1},-g_k+\sigma x_k,x_k), U) = 0.
  \end{aligned}
\end{equation}
	Consequently, there exists $k_0>0$, such that for $k>k_0$, $\text{dist}((x_k,-g_k+\sigma x_k,x_k), U)<\epsilon$ and $\bar{L}< L(x_{k+1},-g_k+\sigma x_k,x_k) < \bar{L} +a$ for $k>k_0$. 
Let $k>k_0$ so that $(x_{k+1},-g_k+\sigma x_k,x_k) \in V$. From~\eqref{eqn:local-affine-KL-1}, we have  
\begin{equation} \label{eqn:local-affine-KL-4}
 \centering
  \begin{aligned}
	  \varphi'( L(x_{k+1},-g_k+\sigma x_k,x_k) - \bar{L})  \cdot \text{dist}(0, \partial L(x_{k+1},-g_k+\sigma x_k,x_k) ) \geq 1, \ \forall k>k_0.
  \end{aligned}
\end{equation}
   To simplify notation, we denote $ y_k = L(x_{k+1},-g_{k}+\sigma x_{k},x_{k}) - \bar{L}$. 
	Using the concavity of $\varphi$ and the KL inequality~\eqref{eqn:local-affine-KL-4}, we have 
\begin{equation} \label{eqn:local-affine-KL-5}
 \centering
  \begin{aligned}
      \relax [\varphi(y_k) &- \varphi(y_{k+1}) ] \cdot \text{dist} (0, \partial L(x_{k+1},-g_k+\sigma x_k,x_k))\\
	  &\geq \varphi'(y_k) (y_k - y_{k+1}) \cdot \text{dist} (0, \partial L(x_{k+1},-g_k+\sigma x_k,x_k))\\
	  &\geq y_k - y_{k+1} = L(x_{k+1},-g_{k}+\sigma x_{k},x_{k}) - L(x_{k+2},-g_{k+1}+\sigma x_{k+1},x_{k+1}).
  \end{aligned}
\end{equation}
	Applying Lemma~\ref{lem:affine-lag-decrease} and Lemma~\ref{lem:lag-subgradient-1} to~\eqref{eqn:local-affine-KL-5}, we have 
\begin{equation} \label{eqn:local-affine-KL-6}
 \centering
  \begin{aligned}
	\relax  [\varphi(y_k) - \varphi(y_{k+1}) ] c_L \norm{d_k} \geq \frac{1}{2}c_d \norm{d_k}^2.
  \end{aligned}
\end{equation}
   Thus,
\begin{equation} \label{eqn:local-affine-KL-7}
 \centering
  \begin{aligned}
	  \norm{d_k} < c_y [\varphi(y_k) - \varphi(y_{k+1}) ], 
  \end{aligned}
\end{equation}
   where $c_y = 2 c_L/c_d$. Summing both sides of~\eqref{eqn:local-affine-KL-7} from $k_0$, we obtain 
\begin{equation} \label{eqn:local-affine-KL-8}
 \centering
  \begin{aligned}
	  \sum_{k=N}^{\infty} \norm{x_{k+1}-x_k} \leq& c_y \sum_{k=N}^{\infty} [\varphi(y_k) - \varphi(y_{k+1}) ]
	  \leq c_y \varphi(y_N) < \infty.
  \end{aligned}
\end{equation}
   Therefore, $\{x_k\}$ is convergent. From Theorem~\ref{thm:simp-KKT}, $\{x_k\}$ converges to a KKT point of~\eqref{eqn:opt0}.
\end{proof}
If the KL properties in Assumption~\ref{assp:affine} is with exponent $\alpha\in [0,1)$, an upper bound of the convergence rate can be shown. An example is given in the following Theorem.
\begin{theorem}\label{thm:local-affine-KL-2}
    Let $b\geq \sigma\geq \rho$ and $2b \geq\sigma+ l$. If~\eqref{eqn:opt1-KLfunction} is a KL function with exponent $\alpha\in (0,\frac{1}{2}]$, then for $k$ large enough,  there exists $q_0\in(0,1)$ and $q_1>0$ such that 
	$\norm{x_k-\bar{x}} \leq q_1 q_0^k$, where $x_k\to\bar{x}$.
\end{theorem}
\begin{proof}
        From Theorem~\ref{thm:local-affine-KL}, we have $x_k\to\bar{x}$.
        Let $L(x_{k+1},-g_k+\sigma x_k, x_k) \to \bar{L}$. Define $ y_k = L(x_{k+1},-g_k+\sigma x_k, x_k) - \bar{L}$.
	Using $\varphi(z) = a_0 z^{1-\alpha}$, by~\eqref{eqn:local-affine-KL-4}, at $(x_{k+1},-g_k+\sigma x_k,x_k)$, we have
\begin{equation} \label{eqn:local-rate-pf-1}
 \centering
  \begin{aligned}
	  a_0 (1-\alpha) y_k^{-\alpha}  \cdot \text{dist}(0, \partial L(x_{k+1},-g_k+\sigma x_k, x_k) ) \geq 1,
  \end{aligned}
\end{equation}
   for $k$ large enough. Therefore, from Lemma~\ref{lem:lag-subgradient-1} and Lemma~\ref{lem:affine-lag-decrease}, 
\begin{equation} \label{eqn:local-rate-pf-2}
 \centering
  \begin{aligned}
	  y_k^{\alpha} \leq a_0 (1-\alpha) \text{dist}(0, \partial L(x_{k+1},-g_k+\sigma x_k, x_k)) 
	  &\leq  a_0 (1-\alpha) c_L \norm{d_k}\\
	  &\leq a_0 (1-\alpha) c_L  [ \frac{2}{c_d}(y_k -y_{k+1}) ]^{\frac{1}{2}}.   
  \end{aligned}
\end{equation}
	Let $c_y = a_0 c_L \sqrt{2/c_d}$.
	Taking square on both sides of~\eqref{eqn:local-rate-pf-2}, we have
 \begin{equation} \label{eqn:local-rate-pf-3}
 \centering
  \begin{aligned}
	  y_k^{2 \alpha} \leq& c_y^2 (1-\alpha)^2 (y_k  - y_{k+1}).   
  \end{aligned}
\end{equation}
	Without losing generality, assume that $y_k<1$. Since $\alpha\in(0,\frac{1}{2}]$, we can write 
 \begin{equation} \label{eqn:local-rate-pf-4}
 \centering
  \begin{aligned}
	  y_{k+1} \leq y_k \leq  y_k^{2 \alpha} \leq& M (y_k  - y_{k+1}),
  \end{aligned}
\end{equation}
  where $M=c_y^2(1-\alpha)^2>0$. Equivalently, 
 \begin{equation} \label{eqn:local-rate-pf-5}
 \centering
  \begin{aligned}
	  y_{k+1}  \leq&  \frac{M}{M+1} y_k.
  \end{aligned}
\end{equation}
	Thus, for $k$ large enough, there exists $y_0>0$ and $q_b\in(0,1)$ so that 
 \begin{equation} \label{eqn:local-rate-pf-5.5}
 \centering
  \begin{aligned}
	  y_{k}  \leq&  y_0 q_b^k.
  \end{aligned}
\end{equation}
	Therefore, $\{L(x_{k+1},-g_k+\sigma x_k,x_k)\}$ converges to $\bar{L}$ $R$-linearly. From~\eqref{eqn:local-affine-KL-7}, we know
 \begin{equation} \label{eqn:local-rate-pf-6}
 \centering
  \begin{aligned}
	  \norm{x_{k+1}-x_k} \leq c_z a_0  [y_k^{1-\alpha} - y_{k+1}^{1-\alpha}],
  \end{aligned}
\end{equation}
 where $c_z= 2c_L/c_d$.
 Summing both sides of~\eqref{eqn:local-rate-pf-6} from $k$ to $k+j$, we have 
 \begin{equation} \label{eqn:local-rate-pf-7}
 \centering
  \begin{aligned}
	  \norm{x_k-x_{k+j}} \leq \sum_{n=k+1}^{n=k+j} \norm{x_{n}-x_{n-1}} \leq \sum_{n=k+1}^{n=k+j} c_z a_0  [y_{n-1}^{1-\alpha} - y_{n}^{1-\alpha}] = c_z a_0 (y_k^{1-\alpha} -y_{k+j}^{1-\alpha}),
  \end{aligned}
\end{equation}
	Given that $x_k\to\bar{x}$ and $y_k \to 0$, letting $j \to \infty$,~\eqref{eqn:local-rate-pf-7} implies
 \begin{equation} \label{eqn:local-rate-pf-8}
 \centering
  \begin{aligned}
	  \norm{x_k-\bar{x}} \leq c_z a_0 y_k^{1-\alpha}. 
  \end{aligned}
\end{equation}
	By~\eqref{eqn:local-rate-pf-5.5}, 
 \begin{equation} \label{eqn:local-rate-pf-9}
 \centering
  \begin{aligned}
	  \norm{x_k-\bar{x}} \leq c_z a_0 (y_0 q_b^k)^{1-\alpha} = c_z a_0 y_0^{1-\alpha} q_b^{k (1-\alpha)}. 
  \end{aligned}
\end{equation}
   Therefore, there exists $q_0 \in (0,1)$ and $q_1>0$ such that $\norm{x_k-\bar{x}} \leq q_1 q_0^k$.
\end{proof}
\begin{remark}
   The rate of convergence result depends on $\alpha\in [0,1)$. For $\alpha \notin (0,\frac{1}{2}]$, readers are referred to~\cite{attouch2009convergence,bolte2014kl,wen2018proximal,liu2019pdcae} for convergence rate upper bound. 
\end{remark}
\subsection{Local convergence with nonlinear constraints}\label{sec:nonlinear-local}
For many smooth SQP methods with line search, the transition to a full step size can occur using second-order correction, or selecting augmented Lagrangian function as the merit function instead to avoid the Maratos effect. The latter has been shown to work for ill-posed or degenerate nonlinear problems using stabilized SQP~\cite{gill2017stabilizedsqp}. The Hessian approximation $B_k$ is also required to be close to the true Hessian at the KKT points~\cite{Nocedal_book} 

Since the objective function is not differentiable, $B_k$ no longer serves as an approximation of the Hessian.
Further, with nonlinear constraints, their values in the merit function are no longer zero, as opposed to the case in Section~\ref{sec:affine-local}. Consequently, a full step is not guaranteed to occur as $k$ increases. 
To take advantage of KL properties, we choose to maintain a positive definite $B_k$ while keeping the exact penalty merit function. We make an additional assumption on the relationship between $B_k$ and the Lagrange multipliers to facilitate a full step size eventually, mimicking the assumptions on $B_k$ in the smooth case. 
The KL property assumption is as follows.
\begin{assumption}\label{assp:KL}
   The potential function $L$ is a KL function. The constraint functions $c_i,i\in\mathcal{E}\cup\mathcal{I}$ are twice continuously differentiable.
\end{assumption}
Define $c_h = m H$, where $H$ is given in Assumption~\ref{assp:boundedHc}. We make the following additional assumption.
\begin{assumption}\label{assp:Bk-c}
   For $k$ large enough, there exists $ c_b \in(0,1)$ such that the positive definite $B_k$ and its corresponding $b$ in Assumption~\ref{assp:Bk} satisfies $c_b b -  \bar{\theta} c_h \geq 0$. 
\end{assumption}
Assumptions~\ref{assp:KL} and~\ref{assp:Bk-c} are assumed valid in this section.
Since Algorithm~\ref{alg:sqp} remains the same, Lemma~\ref{lem:line-search-eqcons} to~\ref{lem:merit-decrease} and Theorem~\ref{thm:simp-KKT} stand. Moreover, the results concerning $L(\cdot,\cdot,\cdot)$ in Section~\ref{sec:affine-local} remain valid as $\bar{c}_i(\cdot)$ are already linearized. 
A full step result similar to that of Lemma~\ref{lem:affine-step-size} can be achieved with Assumption~\ref{assp:Bk-c}.  
\begin{lemma}\label{lem:step-size}
   Let $b\geq \sigma\geq\rho$. The line search condition~\eqref{eqn:line-search-cond} is satisfied with $\alpha_k=1$ and some $\eta \in(0,1)$ for $k$ large enough.
\end{lemma}
\begin{proof}
    The proof is similar to that of Lemma~\ref{lem:affine-step-size}. Similar to~\eqref{eqn:sqp-affine-step-size-pf-1}, we can write 
    \begin{equation} \label{eqn:sqp-step-size-pf-1} 
 \centering
  \begin{aligned}
	  \phi&(x_k,\bar{\theta}) - \phi(x_{k+1},\bar{\theta})  = f(x_k) -f(x_{k+1})+ \bar{\theta} v(x_k)-\bar{\theta}  v(x_{k+1}) \\
			 \geq & \alpha_k d_k^TB_kd_k - \frac{\rho}{2} \alpha_k^2\norm{d_k}^2- \alpha_k \sum_{i\in \mathcal{E}\cup\mathcal{I}} \lambda_i^{k+1} \nabla c_i(x_k)^T d_k + \bar{\theta} v(x_k)-\bar{\theta} v(x_{k+1}) \\
  \end{aligned}
\end{equation}
   Applying Lemma~\ref{lem:line-search-eqcons},~\ref{lem:line-search-ineqcons} with $c_h=c_h^e+c_h^i$ and $b\geq\rho$, we have 
    \begin{equation} \label{eqn:sqp-step-size-pf-2} 
 \centering
  \begin{aligned}
	  \phi(x_k,\bar{\theta}) - \phi(x_{k+1},\bar{\theta})  
			\geq& \frac{1}{2}\alpha_k d_k^TB_kd_k  -\frac{1}{2}\bar{\theta} \alpha_k^2 c_h \norm{d_k}^2  
			= \frac{1}{2}d_k^T (\alpha_k B_k  -\bar{\theta} \alpha_k^2 c_h I) d_k.\\
  \end{aligned}
\end{equation}
   From Assumption~\ref{assp:Bk-c}, if $\alpha_k=1$,  
    \begin{equation} \label{eqn:sqp-step-size-pf-3} 
 \centering
  \begin{aligned}
	  \phi(x_k,\bar{\theta}) - \phi(x_{k+1},\bar{\theta})  
                        \geq&   \frac{1}{2}(1-c_b) d_k^T B_k d_k\\
   \end{aligned}
\end{equation}
  Thus, $\alpha_k$ meets the line search condition~\eqref{eqn:line-search-cond} if $\eta<(1-c_b)$.
\end{proof}
The sufficient decrease result for~\eqref{eqn:opt1-KLfunction} follows. 
\begin{lemma}\label{lem:lag-decrease}
   Let $2b\geq l+\sigma$ and $b\geq \sigma\geq\rho$. If we choose $l>c_b b$, then for $k$ large enough, there exists constants $c_d>0$ such that 
   \begin{equation} \label{eqn:lag-0}
         \centering
         \begin{aligned}
	L(x_{k}, -g_{k-1}+\sigma x_{k-1} , x_{k-1})-L(x_{k+1}, -g_{k}+\sigma x_{k},x_{k}) \geq  \frac{1}{2} c_d \norm{d_{k-1}}^2.
      \end{aligned}
     \end{equation}
\end{lemma}
\begin{proof}
  Similar to~\eqref{eqn:affine-lag-4} in the proof of Lemma~\ref{lem:affine-lag-decrease}, we can arrive at
      \begin{equation} \label{eqn:lag-decrease-1}
         \centering
         \begin{aligned}
      L(&x_{k+1}, -g_k+\sigma x_k , x_k)-L(x_{k}, -g_{k-1}+\sigma x_{k-1},x_{k-1})\\
             %&\leq 
      %-d_k^T B_k d_k +\sum_{i\in\mathcal{E}\cup\mathcal{I}}\lambda^{k+1}_i\nabla c_i(x_k)^T d_k+\frac{1}{2}(l+\sigma)\norm{d_k}^2-\frac{1}{2}(l+\sigma-\rho)\norm{d_{k-1}}^2\\
        &\leq -b \norm{d_k}^2+\frac{1}{2}(l+\sigma)\norm{d_k}^2-\frac{1}{2} (l+\sigma-\rho)\norm{d_{k-1}}^2 - \sum_{i\in\mathcal{E}\cup\mathcal{I}}\lambda^{k+1}_i c_i(x_k). \\
      \end{aligned}
     \end{equation}
     Assumption~\ref{assp:boundedHc} and third line in~\eqref{eqn:sqp-penal-KKT-1} imply
      \begin{equation} \label{eqn:lag-decrease-2}
         \centering
         \begin{aligned}
            \norm{c_i(x_k) - c_i(x_{k-1}) -\nabla c_i(x_{k-1})^T d_{k-1}} 
             =\norm{c_i(x_k)}\leq \frac{1}{2}H \norm{d_{k-1}}^2, \ i\in\mathcal{E}&,\\
            c_i(x_k) \geq c_i(x_{k-1}) +\nabla c_i(x_{k-1})^T d_{k-1} - \frac{1}{2}H \norm{d_{k-1}}^2\geq - \frac{1}{2}H \norm{d_{k-1}}^2, \ i \in\mathcal{I}&.
      \end{aligned}
     \end{equation}
    Notice that $\lambda_i^k\geq 0$ for $i\in\mathcal{I}$. Using~\eqref{eqn:lag-decrease-2} and Lemma~\ref{lem:lambda-sqp-prop},~\eqref{eqn:lag-decrease-1} becomes  
      \begin{equation} \label{eqn:lag-decrease-3}
         \centering
         \begin{aligned}
           L(&x_{k+1}, -g_k+\sigma x_k , x_k)-L(x_{k}, -g_{k-1}+\sigma x_{k-1},x_{k-1})\\
             &\leq -\frac{1}{2}(2b-l-\sigma) \norm{d_k}^2-\frac{1}{2}(l+\sigma-\rho)\norm{d_{k-1}}^2 + \sum_{i\in\mathcal{E}\cup\mathcal{I}} |\lambda^{k+1}_i| \frac{1}{2} H \norm{d_{k-1}}^2 \\
           &\leq -\frac{1}{2}(2b-l-\sigma) \norm{d_k}^2-\frac{1}{2}(l+\sigma-\rho-\bar{\theta}c_h)\norm{d_{k-1}}^2.\\
         \end{aligned}
     \end{equation}
   From Assumption~\ref{assp:Bk-c}, $2b\geq l+\sigma$, $b\geq \sigma\geq \rho$ and $l>c_b b$, the proof is complete.
\end{proof}
Lemma~\ref{lem:lag-subgradient-1} is still valid with similar proofs, \textit{i.e.}, the subgradient of $L(\cdot,\cdot,\cdot)$ satisfies~\eqref{eqn:aux-sub-1}.
The convergence of $\{x_k\}$ follows.
\begin{theorem}\label{thm:local-KL}
 Given Assumption~\ref{assp:KL} and~\ref{assp:Bk-c}, let $2b\geq l+\sigma$, $l/c_b>b\geq \sigma\geq \rho$.  The sequence $\{x_k\}$ generated by Algorithm~\ref{alg:sqp} converges 
    to a KKT point of~\eqref{eqn:opt0}. 
\end{theorem}

Next, a result similar to Theorem~\ref{thm:local-affine-KL-2} is given below. 
\begin{theorem}\label{thm:local-rate-2}
	Given Assumption~\ref{assp:KL} and~\ref{assp:Bk-c}, let $2b\geq l+\sigma$ and $l/c_b>b\geq \sigma\geq \rho$. Suppose the potential function $L$~\eqref{eqn:opt1-KLfunction} satisfies KL property with exponent $\alpha \in [0,1)$.
	If $\alpha \in (0,\frac{1}{2}]$, then there exists $q_0\in(0,1)$ and $q_1>0$ such that for $k$ large enough, 
	$\norm{x_k-\bar{x}} \leq q_1 q_0^k$, where $x_k\to\bar{x}$.
\end{theorem}
The proof of Theorem~\ref{thm:local-KL} and~\ref{thm:local-rate-2} are similar to those of Theorem~\ref{thm:local-affine-KL} and~\ref{thm:local-affine-KL-2} and are thus omitted.

%Other popular assumptions used for local convergence not discussed in this paper include the subdifferential error bound 
%and step-size error bound, which are closely related to KL-property~\cite{}.

%% file: Sections_journal/Examples.tex
\section{Numerical example}\label{sec:exp}
The numerical example we present is a regularized SCACOPF problem that is decomposed into two-stage optimization problems, where 
the second-stage problems are also called contingency problems. The nonsmooth part of the first-stage objective consists of value functions of the second-stage problems. All the objective and constraint functions in both the first- and second-stage problems are analytic. As a practical problem, the optimization variables are all bounded.
	The complete mathematical formulation is complex but the first-stage problem fits in the form of~\eqref{eqn:opt0}.
	Details of the 	problem setup can be seen in~\cite{petra_21_gollnlp}. The number of second-stage/contingency problems is 100.

 The first- and second-stage problems are coupled through linear constraints in $x$ in the second-stage problem. 
Using a quadratic penalty of the coupling constraints in the second-stage problems, $f$ becomes upper-$\Ctwo$ and the problem is referred to as the regularized SCACOPF, in contrast to the non-regularized one. 
Algorithm~\ref{alg:sqp} is applied to solve the regularized SCACOPF. In~\cite{wang2022}, the effectiveness of the regularized problem for the non-regularized one is demonstrated. Hence, our experiment would focus on the regularized problem itself. 
We note that the global and local convergence theories established in this paper can be readily applied.

The network data used in this example is from the ARPA-E Grid Optimization competition~\cite{petra_21_gollnlp}.
The quadratic penalty parameter in the second-stage contingency problem is set to $10^9$. 
To best illustrate the practical results of the local convergence analysis, a fixed $B_k = 10^6 I$, $I$ being the identity matrix, is used after 100 iterations. 
The algorithm is run for $300$ iterations where the stopping criterion of $\epsilon=10^{-8}$ is achieved already. 
The objective plot is given in Figure~\ref{fig:scacopf_obj}.
\begin{figure}
  \centering
  \includegraphics[width=0.8\textwidth]{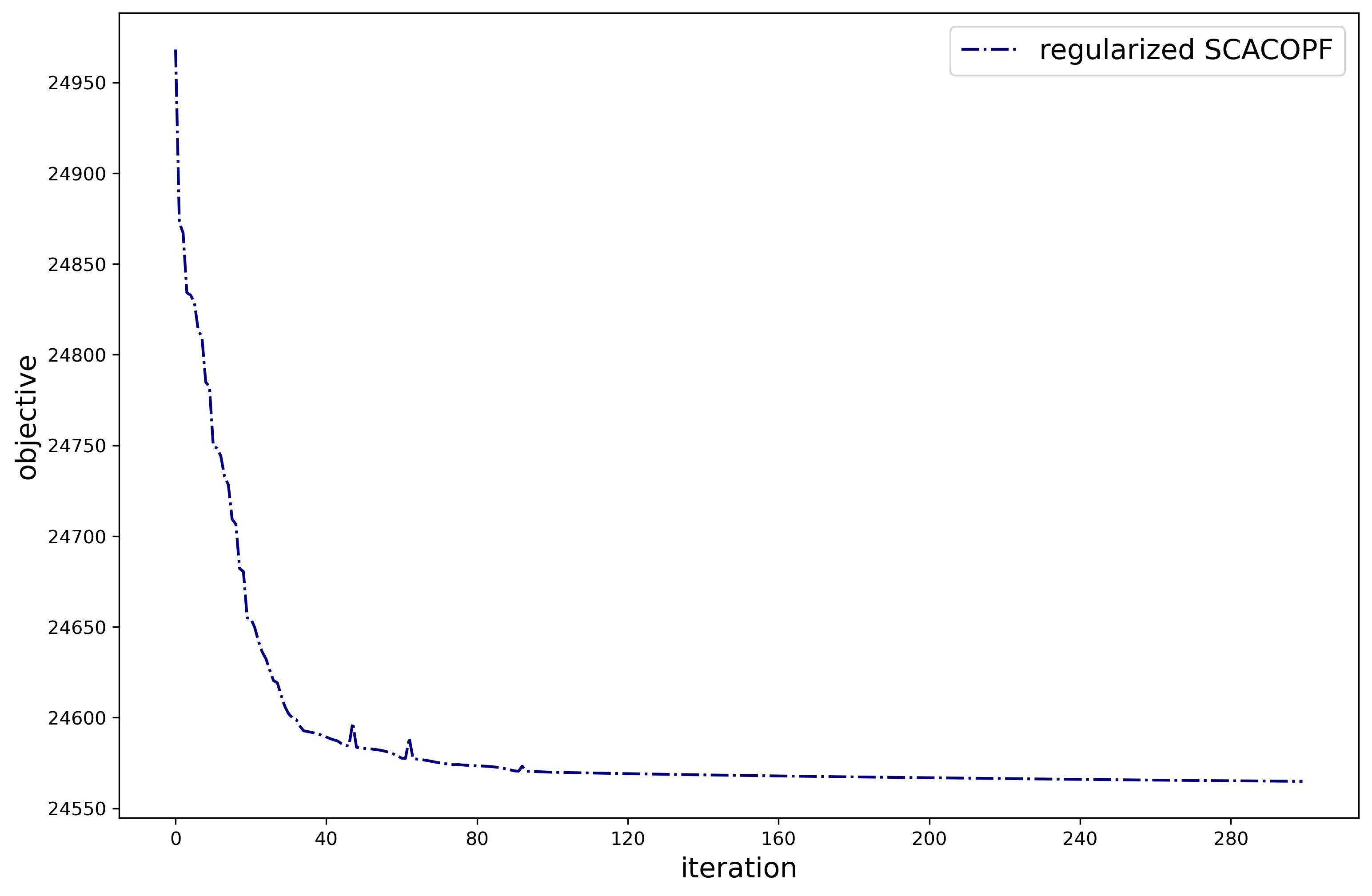}
	\caption{Objective value \textit{v.s.} iterations}
\label{fig:scacopf_obj}
\end{figure}

The final iterate $x_{300}$ is chosen as the true solution for the regularized SCACOPF and the error in $x$ is defined as $\norm{x_k-x_{300}}$. We take the logarithm of the error in $x$ and plot against the number of iterations from 0 to 250 in Figure~\ref{fig:scacopf_local}. 
\begin{figure}
  \centering
  \includegraphics[width=0.8\textwidth]{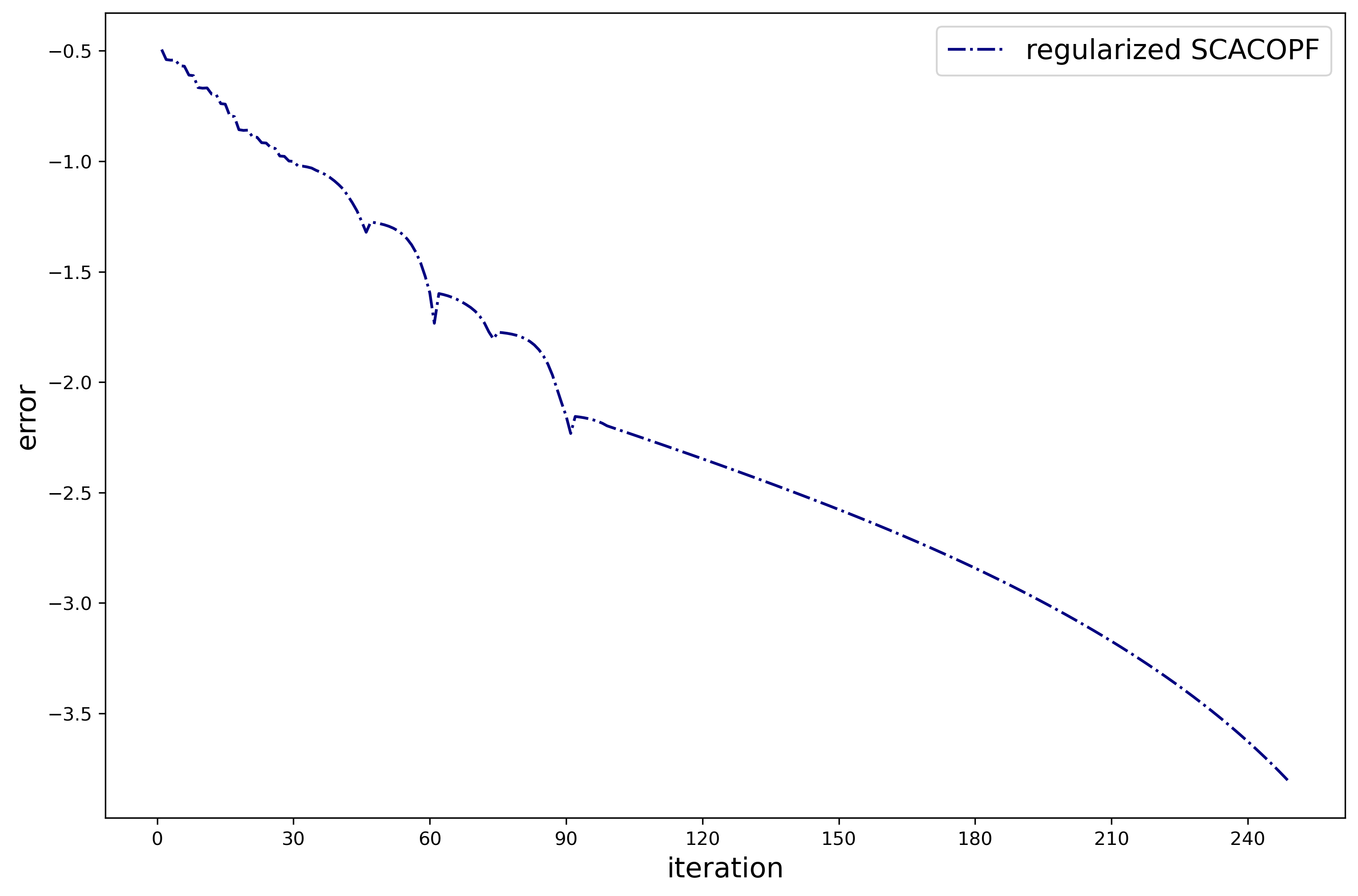}
	\caption{Logarithm error in $x$ \textit{v.s.} iterations}
\label{fig:scacopf_local}
\end{figure}
From Figure~\ref{fig:scacopf_local}, the local convergence behavior for $k>100$ follows Theorem~\ref{thm:local-KL} and~\ref{thm:local-rate-2}. While we do not know the value of exponent $\alpha$ for KL properties, the rate of the decrease of logarithm error can be seen to be bounded above by some linear functions.
Since iteration $300$ is chosen as the true solution, the error $\norm{x_k-x_{300}}$ decreases even faster after 250 iterations and thus clearly observes our analysis.

%% file: Sections_journal/Conclusion.tex
\section{\normalsize Conclusions}\label{sec:con}
In this paper, we have proposed a classic SQP algorithm for optimization problems
with upper-$\Ctwo$ objectives, which exist in many applications, particularly two-stage optimization problems.
The proposed algorithm is subsequentially convergent under reasonable conditions.
Further, the local convergence behavior of the algorithm is analyzed based on the well-known KL properties.  
This important assumption allows us to prove convergence rate upper bound, similar to recent results in DCAs, depending on the specific forms of the KL properties. Importantly, with the help of subanalytic sets and functions, a large number of optimization problems with upper-$\Ctwo$ objectives enjoy a KL potential function.  
These problems include our target application, SCACOPF problems, which we demonstrate as a numerical experiment.  
Finally, we note that the algorithm has been implemented on parallel computing platforms for power grid optimization problems, and has shown significant potential for computational scalability~\cite{wang2021}.